\documentclass[12pt,twoside]{article}
\usepackage{mathrsfs}
\usepackage{amsfonts}
\usepackage{multirow}
\usepackage[dvips]{graphicx}
\usepackage{mathrsfs}
\usepackage{amsfonts}
\usepackage{graphicx}
\usepackage{subfigure}
\usepackage{bbm}
\usepackage{mathrsfs}
\usepackage{dsfont}
\usepackage{mathrsfs}
\usepackage{amsmath}
\usepackage{amssymb}
\usepackage{amsfonts}

\usepackage{amsmath,amssymb}
\usepackage{mathrsfs,color}
\newenvironment{proof}[1][Proof]{\noindent\textbf{#1.} }{\ \rule{0.5em}{0.5em}}
\newtheorem{theorem}{Theorem}[section]
\newtheorem{definition}{Definition}[section]

\newtheorem{example}{Example}[section]
\newtheorem{lemma}[theorem]{\quad Lemma}
\numberwithin{figure}{section} \numberwithin{equation}{section}
\makeatletter 
\newcommand{\bm}[1]{\mbox{\boldmath{$#1$}}}
\makeatletter
\newcommand{\allmodesymb}[2]{\relax\ifmmode{\mathchoice
{\mbox{\fontsize{\tf@size}{\tf@size}#1{#2}}}
{\mbox{\fontsize{\tf@size}{\tf@size}#1{#2}}}
{\mbox{\fontsize{\sf@size}{\sf@size}#1{#2}}}
{\mbox{\fontsize{\ssf@size}{\ssf@size}#1{#2}}}}
\else
\mbox{#1{#2}}\fi}
\makeatother

\usepackage{amsmath}
\begin{document}
\date{}
\author{Hengfei Ding$^{1,}$\thanks{E-mail:
dinghf05@163.com}
\;\;Changpin
Li$^{2,}$\thanks{E-mail:
lcp@shu.edu.cn}
\\
 \small \textit{1. School of Mathematics and Statistics, Tianshui
Normal University, Tianshui 741001, China}\\
\small \textit{2. Department of Mathematics, Shanghai University,
Shanghai 200444, China }
} \vspace{0.2cm}
\title{Fractional-compact numerical algorithms for Riesz spatial fractional reaction-dispersion equations
\thanks{The work was partially supported by the National
Natural Science Foundation of China under Grant Nos. 11372170, 11561060 and 11671251.
}}\maketitle \vspace{0.2 cm}
 \hrulefill

\begin{abstract}

It is well known that using high-order numerical algorithms to solve fractional differential
 equations leads to almost the same
computational cost with low-order ones but the accuracy (or convergence order) is greatly improved,
due to the nonlocal properties of fractional operators. Therefore, developing some high-order numerical
approximation formulas for fractional derivatives play a more important role in numerically solving
fractional differential equations. This paper focuses on constructing (generalized) high-order
 fractional-compact numerical approximation
 formulas for Riesz derivatives. Then we apply the developed formulas to
the one- and two-dimension Riesz spatial fractional reaction-dispersion equations.
The stability and convergence of the derived numerical
algorithms are strictly studied by using the energy analysis method.
Finally, numerical simulations are given to demonstrate the efficiency and convergence orders
of the presented numerical algorithms.
\\\vspace{0.2 cm}
\vspace{0.2 cm}\\
\end{abstract}
\hrulefill

 \section{Introduction}\label{sec:1}

 Riesz derivative $\displaystyle \frac{\partial^\alpha
u(x)}{\partial|x|^\alpha}$ with derivative order $\alpha\in(1,2)$
 is defined by \cite{Gorenflo,Samko}
\begin{equation}\label{eq1}
\displaystyle \frac{\partial^\alpha
u(x)}{\partial{|x|^\alpha}}=\displaystyle -\frac{1}{2\cos\left(\frac{\pi\alpha}{2}\right)}\left
(\,_{RL}D_{a,x}^\alpha+\,_{RL}D_{x,b}^\alpha\right)u(x),\;1<\alpha<2,\;x\in(a,b),
\end{equation}
where
$\,_{RL}D_{a,x}^\alpha$ denotes the left Riemann-Liouville derivative
$$\displaystyle \,_{RL}{D}_{a,x}^{\alpha}u(x)=\left\{
\begin{array}{lll}
\displaystyle\frac{1}{\Gamma(2-\alpha)}\frac{\textmd{d}^2}{\textmd{d}
x^2}
 \int_{a}^{x}\frac{u(s)\textmd{d}s}{(x-s)^{\alpha-1}},\;\;\;\; 1<\alpha<2,\vspace{0.4 cm}\\
\displaystyle\frac{\textmd{d}^{2} u(x)}{\textmd{d} x^{2}},\;\;\;\;\alpha=2,\vspace{0.2 cm}\\
\end{array}\right.
$$
and $\,_{RL}D_{x,b}^\alpha$
the right Riemann-Liouville derivative
$$\displaystyle \,_{RL}{D}_{x,b}^{\alpha}u(x)=\left\{
\begin{array}{lll}
\displaystyle\frac{1}{\Gamma(2-\alpha)}\frac{\textmd{d}^2}{\textmd{d}
x^2}
 \int_{x}^{b}\frac{u(s)\textmd{d}s}{(s-x)^{\alpha-1}},\;\;\;\; 1<\alpha<2,\vspace{0.4 cm}\\
\displaystyle\frac{\textmd{d}^{2} u(x)}{\textmd{d} x^{2}},\;\;\;\;\alpha=2.\vspace{0.2 cm}\\
\end{array}\right.
$$
The Riesz fractional derivative has been shown to be a suitable tool for modeling L\'{e}vy
 flights whose second moments diverge. But the fractional moments exist. They are often used to
  analyze the diffusion behaviors of particles \cite{Felmer,Longhi,Secchi,Stickler}.
  So the Riesz spatial fractional reaction-dispersion equations have attracted increasing interest.
  However, almost
  all of the fractional differential equations cannot
be obtained analytical solutions, hence, more and more studies focus on their numerical
solutions in recent decades \cite{Alikhanov,Baeumer,Baleanu,Cao,Dimitrov1,Dimitrov2,Gao,Garrappa,Jin,Li,Pagnini,Wang,Yan,Ye}.
 For the numerical algorithms
of such problems, the key step is to construct efficient approximation
 formula for the Riesz (or Riemann-Liouville) derivatives. Up to now, there have existed
  some difference formulas to approximate the Riemann-Liouville derivatives and
 Riesz derivatives. One of the most popular approximation is the first-order (shifted) Gr\"{u}nwald-Letnikov
 formula \cite{Meerschaert,Tadjeran}. Recently, some high-order formulas have been constructed based on
(shifted) Gr\"{u}nwald-Letnikov
 formula, for examples, second-order and third-order weighted and shifted Gr\"{u}nwald difference schemes \cite{Tian,Zhou}.
Some other studies have been also developed to approximate the Riesz derivatives with the help of the fractional centered operator
  \cite{Ortigueira}. \c{C}elik and Duman \cite{Celik} studied the convergence order of this approximation and applied it
   to Riesz spatial fractional
diffusion equations. Later on, Shen et al. applied it to Riesz spatial fractional advection-dispersion
 equation and got a weighted difference scheme \cite{Shen}.
Ding and Li \cite{Ding1,Ding2,Ding3} constructed a series of high-order algorithms for Riemann-Liouville (Riesz)
 derivatives and applied them to the different types of fractional
differential equations. It is worth mentioning that we proposed a kind of higher-order numerical approximate formulas
 for Riemann-Liouville (or Riesz) derivatives
based on the novel generating functions in \cite{Ding4}, and obtained a unconditionally stable difference scheme
 where the 2nd-order approximation formula was applied to Riesz spatial fractional advection-diffusion equations.
 In addition, other methods, such as L1/L2 approximation methods \cite{Yang}, (improved) matrix transform methods \cite{Ilic,Zhang},
  etc., were also adopted to approximate the Riemann-Liouville (or Riesz) derivatives.

The purpose of this paper is
to develop several robust and efficient high-order fractional-compact numerical
approximation formulas for Riesz derivatives.
The remainder of the paper is organized as follows. In Section 2, we derive two kinds of 3rd-order
numerical approximate formulas for the Riesz derivatives.
In Section 3, the generalized high-order numerical approximation formulas and their fractional-compact
forms for Riemann-Liouville (also Riesz) derivatives are also constructed.
 In Sections 4 and 5, one of the 3rd-order schemes is applied
to solving one- and two-dimension Riesz spatial fractional reaction-dispersion equations, respectively.
The stability and convergence
analyses of the presented difference schemes are also studied.
In Section 6, numerical examples are carried out to confirm the theoretical results and show
the efficiency of the proposed schemes. Finally, some conclusions are included in Section 7.

\section{Fractional-compact numerical approximation formulas }\label{sec:2}

In this section, we establish fractional-compact numerical schemes for Riesz derivatives, where
the ideas and techniques are also suitable for Riemann-Liouville derivatives.

 \subsection{The third-order fractional-compact formula I}\label{subsec:1.1}

 Firstly, we define the following difference operators,
$$
\begin{array}{lll}
\displaystyle \,^{L}\mathcal{B}_{2}^{\alpha}u(x)
=\frac{1}{h^{\alpha}}\sum\limits_{\ell=0}^{\infty}
\kappa_{2,\ell}^{(\alpha)}u\left(x-(\ell-1)h\right),
\end{array}
$$
and
$$
\begin{array}{lll}
\displaystyle\,^{R}\mathcal{B}_{2}^{\alpha}u(x)
=\frac{1}{h^{\alpha}}\sum\limits_{\ell=0}^{\infty}
\kappa_{2,\ell}^{(\alpha)}u\left(x+(\ell-1)h\right).
\end{array}
$$
Here,
the coefficients read as \cite{Ding4}
\begin{equation}\label{eq2}
 \displaystyle\kappa_{2,\ell}^{(\alpha)}=
\left(\frac{3\alpha-2}{2\alpha}\right)^{\alpha}\sum\limits_{m=0}^{\ell}
\left(\frac{\alpha-2}{3\alpha-2}\right)^m\varpi_{1,m}^{(\alpha)}
\varpi_{1,\ell-m}^{(\alpha)},\;\;\ell=0,1\ldots
\end{equation}
They can be obtained by the associate generating function
 $$W_{2}(z)=\left(\frac{3\alpha-2}{2\alpha}-\frac{2(\alpha-1)}{\alpha}z+\frac{\alpha-2}{2\alpha}z^2\right)^{\alpha},$$
i.e.,
$$
\begin{array}{lll}
\displaystyle\left(\frac{3\alpha-2}{2\alpha}-\frac{2(\alpha-1)}{\alpha}z+\frac{\alpha-2}{2\alpha}z^2\right)^{\alpha}=
\sum\limits_{\ell=0}^{\infty}\kappa_{2,\ell}^{(\alpha)}z^\ell,\;|z|<1.
\end{array}
$$

In particular, the expressions $\varpi_{1,m}^{(\alpha)}$ in (\ref{eq2}) are the coefficients of the power series expansion
 of function $(1-z)^{\alpha}$ for $|z|<1$. They can be computed recursively
$$
\begin{array}{lll}
\displaystyle \varpi_{1,0}^{(\alpha)}=1,\;\varpi_{1,m}^{(\alpha)}
=\left(1-\frac{1+\alpha}{m}\right)\varpi_{1,m-1}^{(\alpha)},\;\;m=1,2,\ldots
\end{array}
$$

Coefficients $\kappa_{2,\ell}^{(\alpha)}\;(\ell=0,1,\ldots)$ can be calculated by the following recursion formulas
$$\left\{
\begin{array}{lll}
 \displaystyle\kappa_{2,0}^{(\alpha)}=
\left(\frac{3\alpha-2}{2\alpha}\right)^{\alpha}, \vspace{0.2 cm}\\
\displaystyle\kappa_{2,1}^{(\alpha)}=
\frac{4\alpha(1-\alpha)}{3\alpha-2}\kappa_{2,0}^{(\alpha)}, \vspace{0.2 cm}\\
\displaystyle\kappa_{2,\ell}^{(\alpha)}=
\frac{4\alpha(1-\alpha)(\alpha-\ell+1)}{(3\alpha-2)\ell}\kappa_{2,\ell-1}^{(\alpha)}
+\frac{(\alpha-2)(2\alpha-\ell+2)}{(3\alpha-2)\ell}\kappa_{2,\ell-2}^{(\alpha)},\;\;\ell\geq2.
\end{array}\right.
$$

And these coefficients $\kappa_{2,\ell}^{(\alpha)}\;(\ell=0,1,\ldots)$ have some important and interesting properties are listed as follows.

\begin{theorem}\cite{Ding4}\label{Th2.1} The coefficients $\kappa_{2,\ell}^{(\alpha)}\;
(\ell=0,1,\ldots)$ have the following properties for $1<\alpha<2$,\vspace{0.2 cm}\\
\verb"(i)"~$\displaystyle\kappa_{2,0}^{(\alpha)}=\left(\frac{3\alpha-2}{2\alpha}\right)^{\alpha}>0$,\;
 $\displaystyle\kappa_{2,1}^{(\alpha)}=\frac{4\alpha(1-\alpha)}{3\alpha-2}\kappa_{2,0}^{(\alpha)}<0$;\vspace{0.2 cm}\\
\verb"(ii)"~$\displaystyle\kappa_{2,2}^{(\alpha)}=\frac{\alpha(8\alpha^3-21\alpha^2+16\alpha-4)}{(3\alpha-2)^2}
\kappa_{2,0}^{(\alpha)}$.\;
$\kappa_{2,2}^{(\alpha)}<0$ if $\alpha\in(1,\alpha_1^{\ast})$, while $\kappa_{2,2}^{(\alpha)}\geq0$ if $\alpha\in[\alpha_1^{\ast},2)$, where
$\displaystyle\alpha_1^{\ast}=\frac{7}{8}+\frac{\sqrt[3]{621+48\sqrt{87}}}{24}+\frac{19}{\sqrt[3]{621+48\sqrt{87}}}\approx1.5333$;\vspace{0.2 cm}\\
\verb"(iii)"~ $\displaystyle\kappa_{2,\ell}^{(\alpha)}\geq0$ if $\ell\geq3$;\vspace{0.2 cm}\\
\verb"(iv)"~ $\displaystyle\kappa_{2,\ell}^{(\alpha)}\sim-\frac{\sin\left(\pi\alpha\right)\Gamma(\alpha+1)}{\pi}\ell^{-\alpha-1} $  as $\ell\rightarrow\infty$;\vspace{0.2 cm}\\
\verb"(v) "~$\displaystyle\kappa_{2,\ell}^{(\alpha)}\rightarrow 0 $ as $\ell\rightarrow\infty$;\vspace{0.2 cm}\\
\verb"(vi) "~$\displaystyle\sum\limits_{\ell=0}^{\infty}\kappa_{2,\ell}^{(\alpha)}=0.$
\end{theorem}

Next, we give the following asymptotic expansion formulas for difference operators
$\,^{L}\mathcal{B}_{2}^{\alpha}$ and $\,^{R}\mathcal{B}_{2}^{\alpha}$, which play an important role
in establishment of high-order algorithms for Riemann-Liouville
derivatives.

\begin{theorem}\cite{Ding4}\label{Th2.2} Let $u(x)\in C^{[\alpha]+n+1}(\mathds{R})$ and all the derivatives of $u(x)$ up to
order $[\alpha]+n+2$ belong to
$L_1(\mathds{R})$.
Then one has
\begin{equation}\label{eq3}
\displaystyle\,^{L}\mathcal{B}_{2}^{\alpha}u(x)
 =\,_{RL}D_{-\infty,x}^{\alpha}u(x)+\sum\limits_{\ell=1}^{n-1}\left(\sigma_{\ell}^{(\alpha)}\,_{RL}D_{-\infty,x}^{\alpha+\ell}u(x)\right)h^{\ell}
+\mathcal{O}(h^n),\;n\geq2,
\end{equation}
and
$$
\begin{array}{lll}
\displaystyle\,^{R}\mathcal{B}_{2}^{\alpha}u(x)
 =\,_{RL}D_{x,+\infty}^{\alpha}u(x)+\sum\limits_{\ell=1}^{n-1}\left(\sigma_{\ell}^{(\alpha)}\,_{RL}D_{x,+\infty}^{\alpha+\ell}u(x)\right)h^{\ell}
+\mathcal{O}(h^n),\;n\geq2,
\end{array}
$$
 hold uniformly on $\mathds{R}$. Here coefficients $\sigma_{\ell}^{(\alpha)}\;(\ell=1,2,\ldots)$ satisfy the equation
$\displaystyle
\frac{\mathrm{e}^{z}}{z^\alpha}W_{2}(\mathrm{e}^{-z})=1+\sum\limits_{\ell=1}^{\infty}\sigma_{\ell}^{(\alpha)}z^\ell
$. Especially, the first three coefficients are
$$
\begin{array}{lll}
\displaystyle \sigma_{1}^{(\alpha)}=0, \;\;\sigma_{2}^{(\alpha)}=
-\frac{2\alpha^2-6\alpha+3}{6\alpha},\;\;
\sigma_{3}^{(\alpha)}=\frac{3\alpha^3-11\alpha^2+12\alpha-4}{12\alpha^2}.
\end{array}
$$
\end{theorem}

Define difference operator $\mathcal{L}$ as
$$
\begin{array}{lll}
\mathcal{L}u(x)=\left(1+\sigma_{2}^{(\alpha)}h^2\delta_x^2\right)u(x),
\end{array}
$$
where $\delta_x^2$ is the second-order central difference operator and is
defined by $\delta_x^2u(x)=\frac{u(x+h)-2u(x)+u(x-h)}{h^2}$. Such $\mathcal{L}$ can be
regarded as a fractional-compact operator by borrowing the appellation of the integer-order case.
Accordingly, the main results are enunciated as follows.

\begin{theorem}\label{Th2.3} Let $u(x)\in C^{[\alpha]+4}(\mathds{R})$ and all the derivatives of $u(x)$
up to order $[\alpha]+5$ belong to
$L_1(\mathds{R})$.
Then there hold
$$
\begin{array}{lll}
\displaystyle\,^{L}\mathcal{B}_{2}^{\alpha}u(x)
 =\mathcal{L}\,_{RL}D_{-\infty,x}^{\alpha}u(x)
+\mathcal{O}(h^3),
\end{array}
$$
and
\begin{equation}\label{eq4}
\displaystyle\,^{R}\mathcal{B}_{2}^{\alpha}u(x)
 =\mathcal{L}\,_{RL}D_{x,+\infty}^{\alpha}u(x)
+\mathcal{O}(h^3),
\end{equation}
uniformly for $x\in\mathds{R}$.
\end{theorem}

\begin{proof}
In equation (\ref{eq3}), taking $n=3$ and noticing
$$
\begin{array}{lll}
\displaystyle\,_{RL}D_{-\infty,x}^{\alpha+2}u(x)
=\frac{\textrm{d}^2}{\textrm{d}x^2}\left(\,_{RL}D_{-\infty,x}^{\alpha}u(x)\right),
\end{array}
$$
then one has
$$
\begin{array}{lll}
\displaystyle\,^{L}\mathcal{B}_{2}^{\alpha}u(x)
 &=&\displaystyle\,_{RL}D_{-\infty,x}^{\alpha}u(x)+\frac{\textrm{d}^2}{\textrm{d}x^2}
 \left(\,_{RL}D_{-\infty,x}^{\alpha}u(x)\right)\sigma_{2}^{(\alpha)}h^{2}
+\mathcal{O}(h^3)\vspace{0.2 cm}\\
&=&\displaystyle
\,_{RL}D_{-\infty,x}^{\alpha}u(x)+\sigma_{2}^{(\alpha)}h^{2}\left(\delta_x^2\,_{RL}D_{-\infty,x}^{\alpha}u(x)+\mathcal{O}(h^2)\right)
+\mathcal{O}(h^3)\vspace{0.2 cm}\\
&=&\displaystyle\mathcal{L}
\,_{RL}D_{-\infty,x}^{\alpha}u(x)
+\mathcal{O}(h^3).
\end{array}
$$

Using the same method, we can prove that equation (\ref{eq4}) holds too. All this completes proof.
\end{proof}

In particular, if function $u(x)$ is defined on a bounded interval $(a, b)$ and satisfies $u(a)=u(b)=0$,
then we can apply zero-extension to $u(x)$ such that it is defined on $\mathds{R}$. Now
we further have the following results.

\begin{theorem}\label{Th2.4} Suppose $u(x)\in C^{[\alpha]+4}((a,b))$, $u(a)=u(b)=0$ and all the derivatives of $u(x)$
up to order $[\alpha]+5$ belong to
$L_1((a,b))$.
Then for any $x\in(a,b)$, one has
\begin{equation}\label{eq5}
\displaystyle\,^{L}\mathcal{A}_{2}^{\alpha}u(x)
 =\mathcal{L}\,_{RL}D_{a,x}^{\alpha}u(x)
+\mathcal{O}(h^3),
\end{equation}
and
\begin{equation}\label{eq6}
\displaystyle\,^{R}\mathcal{A}_{2}^{\alpha}u(x)
 =\mathcal{L}\,_{RL}D_{x,b}^{\alpha}u(x)
+\mathcal{O}(h^3).
\end{equation}
Here, operators $\,^{L}\mathcal{A}_{2}^{\alpha}$ and $\,^{R}\mathcal{A}_{2}^{\alpha}$ are respectively defined as follows,
$$
\begin{array}{lll}
\displaystyle
\,^{L}\mathcal{A}_{2}^{\alpha}u(x)
=\frac{1}{h^{\alpha}}\sum\limits_{\ell=0}^{[\frac{x-a}{h}]+1}
\kappa_{2,\ell}^{(\alpha)}u\left(x-(\ell-1)h\right),
\end{array}
$$
and
$$
\begin{array}{lll}
\displaystyle
\,^{R}\mathcal{A}_{2}^{\alpha}u(x)
=\frac{1}{h^{\alpha}}\sum\limits_{\ell=0}^{[\frac{b-x}{h}]+1}
\kappa_{2,\ell}^{(\alpha)}u\left(x+(\ell-1)h\right).
\end{array}
$$
\end{theorem}

Finally, combing equations (\ref{eq1}), (\ref{eq5}) with (\ref{eq6}) gives
 a 3rd-order fractional-compact numerical approximation formula for Riesz derivatives,
\begin{equation}\label{eq7}
\displaystyle \mathcal{L}\frac{\partial^\alpha
u(x)}{\partial{|x|^\alpha}}= -\frac{1}{2\cos\left(\frac{\pi\alpha}{2}\right)}\left(
\,^{L}\mathcal{A}_{2}^{\alpha}u(x)+\,^{R}\mathcal{A}_{2}^{\alpha}u(x)\right)+\mathcal {O}(h^3),\;1<\alpha<2.
\end{equation}

{\it{\bf Remark 1:} When $\alpha = 2$, we easily know that $\sigma_{2}^{(\alpha)}=\frac{1}{12}$ and $\sigma_{3}^{(\alpha)}=0$,
then (\ref{eq7}) becomes the following classical fourth-order
compact formula for the second order derivative $\frac{\textmd{d}^2
u(x)}{\textrm{d}x^2}$, that is
$$
\begin{array}{lll}
\displaystyle
\frac{\textrm{d}^2
u(x)}{\textrm{d}x^2}
=\left(1+\frac{h^2}{12}\delta_x^2\right)^{-1}\delta_x^2u(x)+\mathcal {O}(h^4).
\end{array}
$$}

 \subsection{The third-order fractional-compact formula II}\label{subsec:1.1}
 In this subsection, we continue to develop another numerical approximate formula for Riesz derivatives.
If we choose a new generating function
 $$\widetilde{W}_{2}(z)=\left(\frac{3\alpha+2}{2\alpha}-\frac{2(\alpha+1)}{\alpha}z+\frac{\alpha+2}{2\alpha}z^2\right)^{\alpha},$$
 then we define the following difference operators,
$$
\begin{array}{lll}
\displaystyle \,^{L}\mathcal{\widetilde{B}}_{2}^{\alpha}u(x)
=\frac{1}{h^{\alpha}}\sum\limits_{\ell=0}^{\infty}
\widetilde{\kappa}_{2,\ell}^{(\alpha)}u\left(x-(\ell+1)h\right),
\end{array}
$$
and
$$
\begin{array}{lll}
\displaystyle\,^{R}\mathcal{\widetilde{B}}_{2}^{\alpha}u(x)
=\frac{1}{h^{\alpha}}\sum\limits_{\ell=0}^{\infty}
\widetilde{\kappa}_{2,\ell}^{(\alpha)}u\left(x+(\ell+1)h\right).
\end{array}
$$
 Here,
the coefficients are
$$
\begin{array}{lll}
 \displaystyle\widetilde{\kappa}_{2,\ell}^{(\alpha)}=
\left(\frac{3\alpha+2}{2\alpha}\right)^{\alpha}\sum\limits_{m=0}^{\ell}
\left(\frac{\alpha+2}{3\alpha+2}\right)^m\varpi_{1,m}^{(\alpha)}
\varpi_{1,\ell-m}^{(\alpha)},\;\;\ell=0,1,\ldots
\end{array}
$$
which have following recursion expressions,
$$\left\{
\begin{array}{lll}
 \displaystyle\widetilde{\kappa}_{2,0}^{(\alpha)}=
\left(\frac{3\alpha+2}{2\alpha}\right)^{\alpha}, \vspace{0.2 cm}\\
\displaystyle\widetilde{\kappa}_{2,1}^{(\alpha)}=
\frac{-4\alpha(1+\alpha)}{3\alpha+2}\widetilde{\kappa}_{2,0}^{(\alpha)}, \vspace{0.2 cm}\\
\displaystyle\widetilde{\kappa}_{2,\ell}^{(\alpha)}=
\frac{-4\alpha(1+\alpha)(\alpha-\ell+1)}{(3\alpha+2)\ell}\widetilde{\kappa}_{2,\ell-1}^{(\alpha)}
+\frac{(\alpha+2)(2\alpha-\ell+2)}{(3\alpha+2)\ell}\widetilde{\kappa}_{2,\ell-2}^{(\alpha)},\;\;\ell\geq2.
\end{array}\right.
$$

In particular, these coefficients $\widetilde{\kappa}_{2,\ell}^{(\alpha)}$ also have the following
properties.

\begin{theorem}\label{Th2.5} The coefficients $\widetilde{\kappa}_{2,\ell}^{(\alpha)}\;(\ell=0,1,\ldots)$ have the following properties for $1<\alpha<2$,\vspace{0.2 cm}\\
\verb"(i)"~$\displaystyle\widetilde{\kappa}_{2,1}^{(\alpha)},\widetilde{\kappa}_{2,3}^{(\alpha)}<0$,\;\vspace{0.2 cm}\\
\verb"(ii)"~
$\widetilde{{\kappa}}_{2,4}^{(\alpha)}<0$ if $\alpha\in(1,\alpha_2^{\ast})$, while $\widetilde{\kappa}_{2,4}^{(\alpha)}\geq0$ if $\alpha\in[\alpha_2^{\ast},2)$, where
$\displaystyle\alpha_2^{\ast}\approx1.4917.$
$\widetilde{{\kappa}}_{2,5}^{(\alpha)}<0$ if $\alpha\in(1,\alpha_3^{\ast})$, while $\widetilde{\kappa}_{2,5}^{(\alpha)}\geq0$ if $\alpha\in[\alpha_3^{\ast},2)$, where
$\displaystyle\alpha_3^{\ast}\approx1.4437$;\vspace{0.2 cm}\\
\verb"(iii)"~ $\displaystyle\widetilde{\kappa}_{2,\ell}^{(\alpha)}\geq0$ if $\ell=0,2, or\; \ell \geq6$;\vspace{0.2 cm}\\
\verb"(iv)"~ $\displaystyle\widetilde{\kappa}_{2,\ell}^{(\alpha)}\sim-\frac{\sin\left(\pi\alpha\right)\Gamma(\alpha+1)}{\pi}\ell^{-\alpha-1} $  as $\ell\rightarrow\infty$;\vspace{0.2 cm}\\
\verb"(v) "~$\displaystyle\widetilde{\kappa}_{2,\ell}^{(\alpha)}\rightarrow 0 $ as $\ell\rightarrow\infty$;\vspace{0.2 cm}\\
\verb"(vi) "~$\displaystyle\sum\limits_{\ell=0}^{\infty}\widetilde{\kappa}_{2,\ell}^{(\alpha)}=0.$
\end{theorem}

\begin{proof}
Here, we only consider (ii) and (iii). The others are easily
obtained by using almost the same proof as that of Theorem \ref{Th2.1}.

\verb"(ii)"~In view of the definitions of
$$
\begin{array}{lll}
\displaystyle\widetilde{{\kappa}}_{2,4}^{(\alpha)}=\frac{\alpha(\alpha-1)g_1(\alpha)}{6(3\alpha+2)^4}\widetilde{\kappa}_{2,0}^{(\alpha)},
\end{array}
$$
and
$$
\begin{array}{lll}
\displaystyle\widetilde{{\kappa}}_{2,5}^{(\alpha)}=\frac{2\alpha(1-\alpha)(\alpha-2)g_2(\alpha)}{15(3\alpha+2)^5}\widetilde{\kappa}_{2,0}^{(\alpha)},
\end{array}
$$
where
$$
\begin{array}{lll}
\displaystyle g_1(\alpha)=64\alpha^6+80\alpha^5
-101\alpha^4-224\alpha^3-88\alpha^2+64\alpha+48,
\end{array}
$$
and
$$
\begin{array}{lll}
\displaystyle g_2(\alpha)=64\alpha^6+80\alpha^5
-101\alpha^4-224\alpha^3-88\alpha^2+64\alpha+48.
\end{array}
$$
Applying the numerical computations, we easily know that there exist solutions
$\alpha_2^{\ast}\approx1.4917$ and $\alpha_3^{\ast}\approx1.4437$ such that $g_1(\alpha^\ast)=0$ and $g_2(\alpha^\ast)=0$, respectively.
Through further analysis, one can obtain the required results.

\verb"(iii)"~ For the cases of $\ell=0,2$, the direct computation leads to the desired results.
For $\ell=6,7$, the expressions of them read as
$$
\begin{array}{lll}
\displaystyle \widetilde{{\kappa}}_{2,6}^{(\alpha)}=\frac{\alpha(\alpha-1)(\alpha-2)g_3(\alpha)}
{90(3\alpha+2)^6}\widetilde{{\kappa}}_{2,0}^{(\alpha)},
\end{array}
$$
and
$$
\begin{array}{lll}
\displaystyle \widetilde{{\kappa}}_{2,7}^{(\alpha)}=-\frac{2\alpha(\alpha-1)(\alpha-2)
(\alpha-3)g_4(\alpha)}{315(3\alpha+2)^7}\widetilde{{\kappa}}_{2,0}^{(\alpha)}.
\end{array}
$$

Note that
$$
\begin{array}{lll}
\displaystyle g_3(\alpha)&=&\displaystyle512\alpha^9-192\alpha^8-2840\alpha^7-1811\alpha^6+3360\alpha^5+4892\alpha^4
+1088\alpha^3\vspace{0.2cm}\\&&-2640\alpha^2-2816\alpha-960\vspace{0.2cm}\\&=&-(\alpha-1)^3\left[\alpha^2(4-\alpha^2)
(512\alpha^2+1344\alpha+1704)+(987\alpha^3+7725\alpha^2\right.\vspace{0.2cm}\\&&\left.+16422\alpha+14482)\right]
-6311(\alpha-1)^2-3861\alpha(\alpha-1)-1407<0,
\end{array}
$$
and
$$
\begin{array}{lll}
\displaystyle g_4(\alpha)&=&512\alpha^{10}-64\alpha^9-2936\alpha^8-2101\alpha^7+3619\alpha^6+4172\alpha^5-2324\alpha^4\vspace{0.2cm}\\&&
-7216\alpha^3-6896\alpha^2-3776\alpha-960\vspace{0.2cm}\\&=&
-(\alpha-1)^3\left[\alpha^3(4-\alpha^2)
(512\alpha^2+1472\alpha+1992)+(285\alpha^4+5292\alpha^3\right.\vspace{0.2cm}\\&&\left.+17145\alpha^2+20152\alpha+29497)\right]
-52076(\alpha-1)^2-39589(\alpha-1)\vspace{0.2cm}\\&&-17970<0,
\end{array}
$$
then one easily obtain that $\widetilde{{\kappa}}_{2,\ell}^{(\alpha)}\geq0$ for $\ell=6,7$.

As for $\ell\geq8$, according to their expressions, one has
$$
\begin{array}{lll}
 \displaystyle\widetilde{\kappa}_{2,\ell}^{(\alpha)}&=&\displaystyle
\left(\frac{3\alpha+2}{2\alpha}\right)^{\alpha}\sum\limits_{m=0}^{\ell}
\left(\frac{\alpha+2}{3\alpha+2}\right)^m\varpi_{1,m}^{(\alpha)}
\varpi_{1,\ell-m}^{(\alpha)}\\\vspace{0.2 cm}
&=&\displaystyle\left(\frac{3\alpha+2}{2\alpha}\right)^{\alpha}\left[\varpi_{1,0}^{(\alpha)}
\varpi_{1,\ell}^{(\alpha)}+\frac{\alpha+2}{3\alpha+2}\varpi_{1,1}^{(\alpha)}
\varpi_{1,\ell-1}^{(\alpha)}+\left(\frac{\alpha+2}{3\alpha+2}\right)^2\varpi_{1,2}^{(\alpha)}
\varpi_{1,\ell-2}^{(\alpha)}\right.\\\vspace{0.2 cm}
&&\displaystyle\left.+\left(\frac{\alpha+2}{3\alpha+2}\right)^{\ell-1}\varpi_{1,1}^{(\alpha)}
\varpi_{1,\ell-1}^{(\alpha)}+\left(\frac{\alpha+2}{3\alpha+2}\right)^{\ell}\varpi_{1,0}^{(\alpha)}
\varpi_{1,\ell}^{(\alpha)}
\right]\\\vspace{0.2 cm}
&&\displaystyle+\left(\frac{3\alpha+2}{2\alpha}\right)^{\alpha}\sum\limits_{m=3}^{\ell-2}
\left(\frac{\alpha+2}{3\alpha+2}\right)^m\varpi_{1,m}^{(\alpha)}
\varpi_{1,\ell-m}^{(\alpha)}\\\vspace{0.2 cm}
\end{array}
$$

$$
\begin{array}{lll}
&=&\displaystyle\left(\frac{3\alpha+2}{2\alpha}\right)^{\alpha}\left[1
-\frac{(\alpha+2)\alpha\ell}{(3\alpha+2)(\ell-\alpha-1)}
+\frac{(\alpha-1)(\ell-1)\alpha\ell}{2(\ell-\alpha-1)(\ell-\alpha-2)}\right.\\\vspace{0.2 cm}
&&\displaystyle\left.\times\left(\frac{\alpha+2}{3\alpha+2}\right)^2
+\left(\frac{\alpha+2}{3\alpha+2}\right)^{\ell}\left(1-\frac{(3\alpha+2)\ell\alpha}{(\alpha+2)(\ell-\alpha-1)}\right)
\right]\varpi_{1,\ell}^{(\alpha)}\\\vspace{0.2 cm}
&&\displaystyle+\left(\frac{3\alpha+2}{2\alpha}\right)^{\alpha}\sum\limits_{m=3}^{\ell-2}
\left(\frac{\alpha+2}{3\alpha+2}\right)^m\varpi_{1,m}^{(\alpha)}
\varpi_{1,\ell-m}^{(\alpha)}\\\vspace{0.2 cm}
&\geq&\displaystyle\left(\frac{3\alpha+2}{2\alpha}\right)^{\alpha}\left[1
-\frac{(\alpha+2)\alpha\ell}{(3\alpha+2)(\ell-\alpha-1)}
+\frac{(\alpha-1)(\ell-1)\alpha\ell}{2(\ell-\alpha-1)(\ell-\alpha-2)}\right.\\\vspace{0.2 cm}
&&\displaystyle\left.\times\left(\frac{\alpha+2}{3\alpha+2}\right)^2
+\left(\frac{\alpha+2}{3\alpha+2}\right)^{8}\left(1-\frac{(3\alpha+2)\ell\alpha}{(\alpha+2)(\ell-\alpha-1)}\right)
\right]\varpi_{1,\ell}^{(\alpha)}\\\vspace{0.2 cm}
&&\displaystyle+\left(\frac{3\alpha+2}{2\alpha}\right)^{\alpha}\sum\limits_{m=3}^{\ell-2}
\left(\frac{\alpha+2}{3\alpha+2}\right)^m\varpi_{1,m}^{(\alpha)}
\varpi_{1,\ell-m}^{(\alpha)}.
\end{array}
$$

Denote
$$
\begin{array}{lll}
\displaystyle S(\alpha,x)&=&\displaystyle 1
-\frac{(\alpha+2)\alpha\ell}{(3\alpha+2)(\ell-\alpha-1)}
+\frac{(\alpha-1)(\ell-1)\alpha\ell}{2(\ell-\alpha-1)(\ell-\alpha-2)}\left(\frac{\alpha+2}{3\alpha+2}\right)^2\vspace{0.4 cm}\\
&&\displaystyle
+\left(\frac{\alpha+2}{3\alpha+2}\right)^{8}\left(1-\frac{(3\alpha+2)\ell\alpha}{(\alpha+2)(\ell-\alpha-1)}\right),\;\;x\geq8,
\end{array}
$$
and
$$
\begin{array}{lll}
\displaystyle P(\alpha,x)&=&\displaystyle2(x-\alpha-1)(x-\alpha-2)S(\alpha,x)\vspace{0.2 cm}\\
&=&\displaystyle 2(x-\alpha-1)(x-\alpha-2)
-2\alpha x(x-\alpha-2)\frac{\alpha+2}{3\alpha+2}\vspace{0.2 cm}\\
&&\displaystyle+\alpha x(\alpha-1)(x-1)\left(\frac{\alpha+2}{3\alpha+2}\right)^2
-2\alpha x(x-\alpha-2)\vspace{0.2 cm}
\\&&\displaystyle\times\left(\frac{\alpha+2}{3\alpha+2}\right)^{7}
+2(x-\alpha-1)(x-\alpha-2)\left(\frac{\alpha+2}{3\alpha+2}\right)^{8}.
\end{array}
$$
Then one has
$$
\begin{array}{lll}
\displaystyle \frac{\partial^2 P(\alpha,x)}{\partial x^2}&=&\displaystyle\frac{2}{(3\alpha+2)^8}
\left(729\alpha^{10}+723\alpha^9-2516\alpha^8+3792\alpha^7\right.\vspace{0.2 cm}\\
&&\displaystyle+37952\alpha^6+80800\alpha^5+89344\alpha^4+59648\alpha^3+25344\alpha^2\vspace{0.2 cm}\\
&&\displaystyle\left.+6912\alpha+1024\right)
\vspace{0.2 cm}\\
&&\displaystyle > 0,
\end{array}
$$
 for $\alpha\in(1,2)$, that is to say, $\frac{\partial P(\alpha,x)}{\partial x}$ is an increasing function, i.e.,
$\frac{\partial P(\alpha,x)}{\partial x}\geq\frac{\partial P(\alpha,x)}{\partial x}|_{x=8}$.
Here,
$$
\begin{array}{lll}
\displaystyle \frac{\partial P(\alpha,x)}{\partial x}|_{x=8}&=&\displaystyle\frac{1}{(3\alpha+2)^8}
\left(15315\alpha^{10}+18225\alpha^9-92564\alpha^8-198960\alpha^7\right. \vspace{0.2 cm}\\
&&\displaystyle+59456\alpha^6+611680\alpha^5+883200\alpha^4+662784\alpha^3\vspace{0.2 cm}\\
&&\displaystyle\left.
+301312\alpha^2+86272\alpha+13312\right)\vspace{0.2 cm}
\\
&&\displaystyle > 0.
\end{array}
$$
It immediately follows that $ P(\alpha,x)$ is also an increasing function with respect to $x$ for $1<\alpha<2$.
So, $P(\alpha,x)> P(\alpha,8)$. Note that
$$
\begin{array}{lll}
\displaystyle P(\alpha,8)=&=&\displaystyle\frac{4}{(3\alpha+2)^8}\left(22247\alpha^{10}+52229\alpha^9-44918\alpha^8-246288\alpha^7\right.
\vspace{0.2 cm}\\
&&\displaystyle-196672\alpha^6+202720\alpha^5+512960\alpha^4+450304\alpha^3+221440\alpha^2 \vspace{0.2 cm}
\\
&&\displaystyle\left.+66816\alpha+10752
\right)\vspace{0.2 cm}\\
&&\displaystyle>0,
\end{array}
$$
for $1<\alpha<2$. Then one easily get
$$
\begin{array}{lll}
\displaystyle S(\alpha,x)=\frac{1}{2(x-\alpha-1)(x-\alpha-2)}P(\alpha,x)>0,
\end{array}
$$
which implies that $\widetilde{\kappa}_{2,\ell}^{(\alpha)}\geq 0$ for $\ell\geq8$. Combining the former analysis again gives
$$
\begin{array}{lll}
 \displaystyle\widetilde{\kappa}_{2,\ell}^{(\alpha)}\geq 0,
\end{array}
$$
for $\ell\geq6$. All this ends the proof.
\end{proof}

Similar to the previous discussion, we can similarly obtain the following results.

\begin{theorem}\label{Th2.6} Let $u(x)\in C^{[\alpha]+n+1}(\mathds{R})$ and all the derivatives of $u(x)$ up to
order $[\alpha]+n+2$ belong to
$L_1(\mathds{R})$.
Then
$$
\begin{array}{lll}
\displaystyle\,^{L}\mathcal{\widetilde{B}}_{2}^{\alpha}u(x)
 =\,_{RL}D_{-\infty,x}^{\alpha}u(x)+\sum\limits_{\ell=1}^{n-1}\left(\widetilde{\sigma}_{\ell}^{(\alpha)}
 \,_{RL}D_{-\infty,x}^{\alpha+\ell}u(x)\right)h^{\ell}
+\mathcal{O}(h^n),\;n\geq2,
\end{array}
$$
and
$$
\begin{array}{lll}
\displaystyle\,^{R}\mathcal{\widetilde{B}}_{2}^{\alpha}u(x)
 =\,_{RL}D_{x,+\infty}^{\alpha}u(x)+\sum\limits_{\ell=1}^{n-1}\left(\widetilde{\sigma}_{\ell}^{(\alpha)}
 \,_{RL}D_{x,+\infty}^{\alpha+\ell}u(x)\right)h^{\ell}
+\mathcal{O}(h^n),\;n\geq2,
\end{array}
$$
 hold uniformly on $\mathds{R}$. Here coefficients $\widetilde{\sigma}_{\ell}^{(\alpha)}\;(\ell=1,2,\ldots)$ satisfy equation
$\displaystyle
\frac{\mathrm{e}^{-z}}{z^\alpha}\widetilde{W}_{2}(\mathrm{e}^{-z})=1+\sum\limits_{\ell=1}^{\infty}\widetilde{\sigma}_{\ell}^{(\alpha)}z^\ell
,\;|z|<1$. Especially, the first three coefficients are explicitly expressed as
$$
\begin{array}{lll}
\displaystyle \widetilde{\sigma}_{1}^{(\alpha)}=0, \;\;\widetilde{\sigma}_{2}^{(\alpha)}=
-\frac{2\alpha^2+6\alpha+3}{6\alpha},\;\;
\widetilde{\sigma}_{3}^{(\alpha)}=\frac{3\alpha^3+11\alpha^2+12\alpha+4}{12\alpha^2}.
\end{array}
$$
\end{theorem}

Define another fractional-compact difference operator $\mathcal{\widetilde{L}}$ as
$$
\begin{array}{lll}
\mathcal{\widetilde{L}}u(x)=\left(1+\widetilde{\sigma}_{2}^{(\alpha)}h^2\delta_x^2\right)u(x),
\end{array}
$$
then the corresponding theorem is stated below.

\begin{theorem}\label{Th2.7}
Let $u(x)\in C^{[\alpha]+4}(\mathds{R})$ and all the derivatives of $u(x)$
up to order $[\alpha]+5$ belong to
$L_1(\mathds{R})$.
Then there hold
$$
\begin{array}{lll}
\displaystyle\,^{L}\mathcal{\widetilde{B}}_{2}^{\alpha}u(x)
 =\mathcal{\widetilde{L}}\,_{RL}D_{-\infty,x}^{\alpha}u(x)
+\mathcal{O}(h^3),
\end{array}
$$
and
$$
\begin{array}{lll}
\displaystyle\,^{R}\mathcal{\widetilde{B}}_{2}^{\alpha}u(x)
 =\mathcal{\widetilde{L}}\,_{RL}D_{x,+\infty}^{\alpha}u(x)
+\mathcal{O}(h^3),
\end{array}
$$
uniformly for $x\in\mathds{R}$.
\end{theorem}
\begin{proof}
The proof is almost the same as that of Theorem \ref{Th2.3}, so we omit the proof here or leave to the
reads as an exercise.
\end{proof}

By the similar technique, another 3rd-order fractional-compact numerical approximation formula for Riesz derivative
 reads as
\begin{equation}\label{eq8}
\displaystyle \mathcal{\widetilde{L}}\frac{\partial^\alpha
u(x)}{\partial{|x|^\alpha}}= -\frac{1}{2\cos\left(\frac{\pi\alpha}{2}\right)}\left(
\,^{L}\mathcal{\widetilde{A}}_{2}^{\alpha}u(x)+\,^{R}\mathcal{\widetilde{A}}_{2}^{\alpha}u(x)\right)+\mathcal {O}(h^3),
\end{equation}
where operators $\,^{L}\mathcal{\widetilde{A}}_{2}^{\alpha}$ and $\,^{R}\mathcal{\widetilde{A}}_{2}^{\alpha}$ are defined by
$$
\begin{array}{lll}
\displaystyle
\,^{L}\mathcal{\widetilde{A}}_{2}^{\alpha}u(x)
=\frac{1}{h^{\alpha}}\sum\limits_{\ell=0}^{[\frac{x-a}{h}]-1}
\widetilde{\kappa}_{2,\ell}^{(\alpha)}u\left(x-(\ell+1)h\right),
\end{array}
$$
and
$$
\begin{array}{lll}
\displaystyle
\,^{R}\mathcal{\widetilde{A}}_{2}^{\alpha}u(x)
=\frac{1}{h^{\alpha}}\sum\limits_{\ell=0}^{[\frac{b-x}{h}]-1}
\widetilde{\kappa}_{2,\ell}^{(\alpha)}u\left(x+(\ell+1)h\right).
\end{array}
$$

\section{Generalized numerical algorithm formulas and their fractional-compact forms}\label{sec:3}

\subsection{Generalized numerical algorithm formulas for Riesz derivatives}\label{sec:3.1}
 At present, for almost all of the numerical algorithms for fractional derivatives,
 they are all through some of the other grid point values to calculate the value of a particular grid point.
 However, sometimes we need to calculate the arbitrary point values. At this time, the existing formulas cannot be used,
 therefore it is necessary to establish some more general numerical algorithm formulas for fractional derivatives. Here,
 we firstly give the more general numerical algorithm formulas for Riemann-Liouville (and Riesz) derivatives.

\begin{theorem}\label{Th3.1}({\bf Generalized numerical approximation formula for Riemann-Liouville derivatives})
 Let $u(x)\in C^{[\alpha]+p+1}(\mathds{R})$ and all the derivatives of $u(x)$
up to order $[\alpha]+p+2$ belong to
$L_1(\mathds{R})$. For any $s\in \mathds{R}$ and set
$$
\begin{array}{lll}
\displaystyle \,^{L}\mathcal{B}_{p,s}^{\alpha}u(x)
=\frac{1}{h^{\alpha}}\sum\limits_{\ell=0}^{\infty}
\mu_{p,\ell}^{(\alpha,s)}u\left(x-(\ell+s) h\right),
\end{array}
$$
and
$$
\begin{array}{lll}
\displaystyle \,^{R}\mathcal{B}_{p,s}^{\alpha}u(x)
=\frac{1}{h^{\alpha}}\sum\limits_{\ell=0}^{\infty}
\mu_{p,\ell}^{(\alpha,s)}u\left(x+(\ell+s) h\right).
\end{array}
$$
Here, the coefficients
$\displaystyle\mu_{p,\ell}^{(\alpha,s)}\;\left(
\ell=0,1,\ldots\right)$ can be determined by the following
generating functions ${ G}_{p,s}(z)$
$$
\begin{array}{lll}
\displaystyle { G}_{p,s}(z)=\left((1-z)+\sum_{k=2}^{p}
\frac{\vartheta_{{k-1},{k-1}}^{(\alpha,s)}}{\alpha}(1-z)^k\right)^{\alpha},
\end{array}
$$
that is,
$$
\begin{array}{lll}
\displaystyle{ G}_{p,s}(z)=\sum\limits_{\ell=0}^{\infty}\mu_{p,\ell}^{(\alpha,s)}z^\ell,\;|z|<1,
\end{array}
$$
where the parameters $\vartheta_{{{k-1},{k-1}}}^{(\alpha,s)}$ $(k=2,3,\ldots)$ can be obtained by the following equation
$$
\begin{array}{lll}
\displaystyle G_{k,s}\left(e^{-z}\right)\frac{e^{-sz}}{z^\alpha}=1-\sum_{\ell=1}^{\infty}\vartheta_{k,\ell}^{(\alpha,s)}z^{\ell},\;\;k=1,2,\ldots
\end{array}
$$

Then the left and right Riemann-Liouville derivative values at any point $x=x_{j}+sh$ can be approximated by
$$
\begin{array}{lll}
\displaystyle \,_{RL}D_{-\infty,x}^{\alpha}u(x)|_{x=x_{j}+sh}
=\,^{L}\mathcal{B}_{p,s}^{\alpha}u(x_j+sh)+\mathcal{O}(h^p),\;j=0,1,\ldots,\;p\geq1,
\end{array}
$$
and
$$
\begin{array}{lll}
\displaystyle \,_{RL}D_{x,+\infty}^{\alpha}u(x)|_{x=x_{j}+sh}
=\,^{R}\mathcal{B}_{p,s}^{\alpha}u(x_j+sh)+\mathcal{O}(h^p),\;j=0,1,\ldots,\;p\geq1,
\end{array}
$$
respectively.
\end{theorem}
\begin{proof}
This theorem can be viewed as the extension of Theorem 4 in \cite{Ding4}. The proof can be finished
 by almost the same method and we omit it here.
\end{proof}

{\it{\bf Remark 2:} The above theorem is called the generalized numerical approximation formula for
 Riemann-Liouville derivatives, due to any point $x=x_{j}+sh, (j=0,1,\ldots)$ on the real axis can be calculated.
 Here, the needed value $x$ can be determined by selecting the appropriate parameter $s$, $x_{j}$ represent the grid point values.}

Accordingly, the generalized numerical algorithms for Riesz derivatives can be obtained by,
\begin{eqnarray}
\displaystyle \frac{\partial^\alpha
u(x)}{\partial{|x|^\alpha}}\left|_{x=x_{j}+sh}\right. =\displaystyle
 -\frac{1}{2\cos\left(\frac{\pi\alpha}{2}\right)}\left[\,^{R}\mathcal{B}_{p,s}^{\alpha}+\,^{L}\mathcal{B}_{p,s}^{\alpha}
\right]u(x_j+sh)\displaystyle+\mathcal {O}(h^p),\;p\geq1.\nonumber
\end{eqnarray}

Below, we only study the cases for $p=2,3,4$ in details. Due to the fact that case $p=1$ is the same
as the Gr\"{u}nwald-Letnikov formula, and the fact that
the cases for $p\geq5$ can be similarly obtained in view of the above theorem, we carefully consider cases with
$p=2,3,4$ as follows.

\verb"(i)"~$p=2$

According to Theorem \ref{Th3.1}, we easily know that the generating function for $p=2$ is
$$
\begin{array}{lll}
\displaystyle { G}_{2,s}(z)=\left((1-z)+
\frac{\alpha+2s}{2\alpha}(1-z)^2\right)^{\alpha},
\end{array}
$$
and the coefficients $\mu_{2,\ell}^{(\alpha,s)}\;(\ell=0,1,\ldots)$ are read as,
\begin{equation*}
 \displaystyle\mu_{2,\ell}^{(\alpha,s)}=
d_{21}^{\alpha}\sum\limits_{m=0}^{\ell}
d_{22}^m\varpi_{1,m}^{(\alpha)}
\varpi_{1,\ell-m}^{(\alpha)},\;\;\ell=0,1\ldots
\end{equation*}
where,
\begin{equation*}
 \displaystyle
 d_{21}=\frac{3\alpha+2s}{2\alpha},\;\;d_{22}=\frac{\alpha+2s}{3\alpha+2s}.
\end{equation*}

Furthermore, the coefficients $\mu_{2,\ell}^{(\alpha)}\;(\ell=0,1,\ldots)$ can be obtained by the
 following recurrence relationships,
$$\left\{
\begin{array}{lll}
 \displaystyle\mu_{2,0}^{(\alpha,s)}&=&\displaystyle
\left(\frac{3\alpha+2s}{2\alpha}\right)^{\alpha},\vspace{0.2 cm}\\
 \displaystyle\mu_{2,1}^{(\alpha,s)}&=&\displaystyle
\frac{-4\alpha(\alpha+s)}{3\alpha+2s}\mu_{2,0}^{(\alpha,s)},\vspace{0.2 cm}\\ \displaystyle
\mu_{2,\ell}^{(\alpha,s)}&=&\displaystyle
\frac{1}{(3\alpha+2s)\ell}\left[-4(\alpha+s)(\alpha-\ell+1)\mu_{2,\ell-1}^{(\alpha,s)}\right.
\vspace{0.2 cm}\\ &&\displaystyle\left.+(\alpha+2s)(2\alpha-\ell+2)\mu_{2,\ell-2}^{(\alpha,s)}
\right],\;\;\ell\geq2.
\end{array}\right.
$$

\verb"(ii)"~$p=3$

For this case, the generating function is given below,
$$
\begin{array}{lll}
\displaystyle { G}_{3,s}(z)=\left((1-z)+
\frac{\alpha+2s}{2\alpha}(1-z)^2+\frac{2\alpha^2+6\alpha s+3s^2}{6\alpha^2}(1-z)^3\right)^{\alpha},
\end{array}
$$
the coefficients $\mu_{3,\ell}^{(\alpha,s)}\;(\ell=0,1,\ldots)$ can be obtained by simple calculations,
\begin{equation*}
 \displaystyle\mu_{3,\ell}^{(\alpha,s)}=
d_{31}^{\alpha}
\sum\limits_{\ell_1=0}^{\ell}
\sum\limits_{\ell_2=0}^{\left[\frac{1}{2}\ell_1\right]}
\frac{\left(-1\right)^{\ell_1+\ell_2}(\ell_1-\ell_2)!}{\ell_2!(\ell_1-2\ell_2)!}d_{32}^{\ell_1-2\ell_2}
d_{33}^{\ell_2}
{\varpi}_{1,\ell-\ell_1}^{(\alpha)}
{\varpi}_{1,\ell_1-\ell_2}^{(\alpha)},\;\ell=0,1,\ldots
\end{equation*}
Here,
$$\begin{array}{lll}
\displaystyle
d_{31}=\frac{11\alpha^2+12\alpha s+3s^2}{6\alpha^2},\;\;
d_{32}=-\frac{7\alpha^2+18\alpha s+6s^2}{11\alpha^2+12\alpha s+3s^2},\;\;\\\vspace{0.2 cm}\displaystyle
d_{33}=\frac{2\alpha^2+6\alpha s+3s^2}{11\alpha^2+12\alpha s+3s^2}.
\end{array}
$$

The recursion relations for coefficients $\mu_{3,\ell}^{(\alpha,s)}\;(\ell=0,1,\ldots)$ read as,
$$\left\{
\begin{array}{lll}
\displaystyle
\mu_{3,0}^{(\alpha,s)}&=&\displaystyle\left(\frac{11\alpha^2+12\alpha s+3s^2}{6\alpha^2}\right)^{\alpha},\;\vspace{0.2 cm}\\
\displaystyle\mu_{3,1}^{(\alpha,s)}&=&\displaystyle-\frac{3\alpha(6\alpha^2+10\alpha s +3s^2)}{11\alpha^2+12\alpha s +3s^2}
\mu_{3,0}^{(\alpha,s)},\;\vspace{0.2 cm}\\
\displaystyle\mu_{3,2}^{(\alpha,s)}&=&\displaystyle\frac{3\alpha}
{2(11\alpha^2+12\alpha s +3s^2)^2}\left(108\alpha^5+360\alpha^4 s-42\alpha^4+408\alpha^3s^2\right.
\vspace{0.2 cm}\\&&\left.-112\alpha^3 s+180\alpha^2 s^3-132\alpha^2 s^2+27\alpha s^4-60\alpha s^3-9s^4
\right)\mu_{3,0}^{(\alpha,s)}
,\vspace{0.2 cm}\\
\displaystyle\mu_{3,\ell}^{(\alpha,s)}&=&\displaystyle\frac{1}{(11\alpha^2+12\alpha s+3s^2)\ell}\left[-3(6\alpha^2+10\alpha s+3s^2)(\alpha-\ell+1)
\mu_{3,\ell-1}^{(\alpha,s)}\right.\vspace{0.2 cm}
\\&&\displaystyle+3(3\alpha^2+8\alpha s+3s^2)(2\alpha-\ell+2)\mu_{3,\ell-2}^{(\alpha,s)}\vspace{0.2 cm}\\&&\left.
-(2\alpha^2+6\alpha s+3s^2)(3\alpha-\ell+3)\mu_{3,\ell-3}^{(\alpha,s)}
\right],\;\ell\geq3.
\end{array}\right.
$$

\verb"(iii)"~$p=4$

As before, we can also easily get the following generating function for $p=4$,
$$
\begin{array}{lll}
\displaystyle { G}_{4,s}(z)&=&\displaystyle\left((1-z)+
\frac{\alpha+2s}{2\alpha}(1-z)^2+\frac{2\alpha^2+6\alpha s+3s^2}{6\alpha^2}(1-z)^3\right.\vspace{0.4 cm}\\
&&\displaystyle\left.+\frac{3\alpha^3+11\alpha^2 s+9\alpha s^2+2s^3}{12\alpha^3}(1-z)^4
\right)^{\alpha}.
\end{array}
$$
By the back-of-the-envelope calculation, we can get the expressions of coefficients
$\mu_{4,\ell}^{(\alpha,s)}\;(\ell=0,1,\ldots)$ as follows,
$$\begin{array}{lll}
\displaystyle \mu_{4,\ell}^{(\alpha,s)}=d_{41}^{\alpha}
\sum\limits_{\ell_1=0}^{\ell}
\sum\limits_{\ell_2=0}^{\left[\frac{2}{3}\ell_1\right]}
\sum\limits_{\ell_3=\max\{0,2\ell_2-\ell_1\}}^{\left[\frac{1}{2}\ell_2\right]}
P(\alpha,\ell_1,\ell_2,\ell_3)
\varpi_{1,\ell-\ell_1}^{(\alpha)}\varpi_{1,\ell_1-\ell_2}^{(\alpha)},
\end{array}
$$
where
$$\begin{array}{lll}
\displaystyle
P(\alpha,\ell_1,\ell_2,\ell_3)=\frac{\left(-1\right)^{\ell_1+\ell_2}\left(\ell_1-\ell_2\right)!}
{\ell_3!\left(\ell_2-2\ell_3\right)!\left(\ell_1+\ell_3-2\ell_2\right)!}d_{42}^{\ell_1+\ell_3-2\ell_2}d_{43}^{\ell_2-2\ell_3}d_{44}^{\ell_3},
\end{array}
$$
and
$$\begin{array}{lll}
\displaystyle
d_{41}=\frac{25\alpha^3+35\alpha^2 s+15\alpha s^2+2s^3}{12\alpha^3},\;\;
d_{42}=-\frac{23\alpha^3+69\alpha^2 s+39\alpha s^2+6s^3}{25\alpha^3+35\alpha^2 s+15\alpha s^2+2s^3},\;\;\vspace{0.2 cm}\\ \displaystyle
d_{43}=\frac{13\alpha^3+45\alpha^2 s+33\alpha s^2+6s^3}{25\alpha^3+35\alpha^2 s+15\alpha s^2+2s^3},\;\;
d_{44}=-\frac{3\alpha^3+11\alpha^2 s+9\alpha s^2+2s^3}{25\alpha^3+35\alpha^2 s+15\alpha s^2+2s^3}.
\end{array}
$$

The recursion formulas of coefficient $\mu_{4,\ell}^{(\alpha,s)}\;(\ell=0,1,\ldots)$ are shown below,
$$
\begin{array}{lll}
\displaystyle\mu_{4,0}^{(\alpha,s)}&=&\displaystyle\left(\frac{25\alpha^3+35\alpha^2 s+15\alpha s^2+2s^3}{12\alpha^3}\right)^{\alpha},\;\vspace{0.2 cm}\\
\displaystyle\mu_{4,1}^{(\alpha,s)}&=&\displaystyle-\frac{2\alpha\left(24\alpha^3+52\alpha^2s+27\alpha s^2+4s^3\right)}{25\alpha^3+35\alpha^2 s+15\alpha s^2+2s^3}\mu_{4,0}^{(\alpha,s)},\;\vspace{0.2 cm}\\
\displaystyle\mu_{4,2}^{(\alpha,s)}&=&\displaystyle\frac{2\alpha}
{(25\alpha^3+35\alpha^2 s+15\alpha s^2+2s^3)^2}\left.(576\alpha^7+2496\alpha^6 s+4000\alpha^5s^2\right.\vspace{0.2 cm}\\
&&\displaystyle+3000\alpha^4s^3 +1145\alpha^3 s^4+216\alpha^2 s^5+16\alpha s^6
-126\alpha^6-441\alpha^5 s\vspace{0.2 cm}\\&&\displaystyle\left.-835\alpha^4 s^2-699\alpha^3s^3-281\alpha^2s^4-54\alpha s^5-4s^6
\right.)\mu_{4,0}^{(\alpha,s)},\vspace{0.2 cm}\\
\displaystyle\mu_{4,3}^{(\alpha,s)}&=&\displaystyle-\frac{2\alpha}
{(25\alpha^3+35\alpha^2 s+15\alpha s^2+2s^3)^3}\left.(
27648\alpha^{11}+179712\alpha^{10}s\right.\vspace{0.2 cm}\\
 &&\displaystyle+482688\alpha^9s^2+699392\alpha^8s^3+602928\alpha^7s^4+323448\alpha^6s^5
 \vspace{0.2 cm}\\&&\displaystyle
 +109062\alpha^5s^6+22488\alpha^4s^7+2592\alpha^3s^8+128\alpha^2s^9
  -18144\alpha^{10} \vspace{0.2 cm}\\&&\displaystyle-102816\alpha^9s-278244\alpha^8s^2-43564\alpha^7s^3-404406\alpha^6s^4  \vspace{0.2 cm}\\
 &&\displaystyle -228726\alpha^5s^5-79722\alpha^4s^6-16740\alpha^3s^7-1944\alpha^2s^8-96\alpha s^9 \vspace{0.2 cm}\\
 &&\displaystyle +5496\alpha^9+17604\alpha^8s+29331\alpha^7s^2+47500\alpha^6s^3+52263\alpha^5s^4 \vspace{0.2 cm}\\
 &&\displaystyle\left.+33528\alpha^4s^5+12591\alpha^3s^6+2748\alpha^2s^7+324\alpha^8+16s^9\right)\mu_{4,0}^{(\alpha,s)},
\end{array}
$$
$$\begin{array}{lll}
\displaystyle\mu_{4,\ell}^{(\alpha,s)}&=\displaystyle\frac{1}
{(25\alpha^3+35\alpha^2 s+15\alpha s^2+2s^3)\ell}\left[-2(24\alpha^3+52\alpha^2s+27\alpha s^2+4s^3)\right.
 \vspace{0.2 cm}\\&\displaystyle\times(\alpha-\ell+1)\mu_{4,\ell-1}^{(\alpha,s)}
+6(6\alpha^3+19\alpha^2s+12\alpha s^2+2s^3)(2\alpha-\ell+2)\mu_{4,\ell-2}^{(\alpha,s)}\vspace{0.2 cm}\\&\displaystyle
-2(8\alpha^3+28\alpha^2s+21\alpha s^2+4s^3)(3\alpha-\ell+3)\mu_{4,\ell-3}^{(\alpha,s)}
\vspace{0.2 cm}\\&\displaystyle\left.+(3\alpha^3+11\alpha^2s+9\alpha s^2+2s^3)(4\alpha-\ell+4)\mu_{4,\ell-4}^{(\alpha,s)}
\right],\;\ell\geq4.
\end{array}
$$

{\it{\bf Remark 3:}  It is a remarkable finding that Lublich's method \cite{Lubich} and
the modified high-order numerical algorithm
 formulas in \cite{Ding4} are the special cases
of Theorem \ref{Th3.1} for $s=0$ and $s=-1$, respectively. }

\subsection{The fractional-compact forms of the generalized numerical algorithms}\label{sec:3.1}

Firstly, the asymptotic expansion formulas of operators
$\,^{L}\mathcal{B}_{p,s}^{\alpha}$ and $\,^{R}\mathcal{B}_{p,s}^{\alpha}$ are listed as follows,
 which are the foundations for the establishment of the fractional-compact forms of the generalized numerical algorithms.

\begin{theorem} Let $u(x)\in C^{[\alpha]+n+1}(\mathds{R})$ and all the derivatives of $u(x)$ up to
order $[\alpha]+n+2$ belong to
$L_1(\mathds{R})$.
Then for any $s\in \mathds{R}$ and $p\in \mathds{N}$, one has
\begin{equation*}
\displaystyle\,^{L}\mathcal{B}_{p,s}^{\alpha}u(x)
 =\,_{RL}D_{-\infty,x}^{\alpha}u(x)+\sum\limits_{\ell=p}^{n-1}\left(\varrho_{\ell}^{(\alpha,s)}
 \,_{RL}D_{-\infty,x}^{\alpha+\ell}u(x)\right)h^{\ell}
+\mathcal{O}(h^n),\;n\geq p+1,
\end{equation*}
and
\begin{equation*}
\displaystyle\,^{R}\mathcal{B}_{p,s}^{\alpha}u(x)
 =\,_{RL}D_{x,+\infty}^{\alpha}u(x)+\sum\limits_{\ell=p}^{n-1}\left(\varrho_{\ell}^{(\alpha,s)}
 \,_{RL}D_{x,+\infty}^{\alpha+\ell}u(x)\right)h^{\ell}
+\mathcal{O}(h^n),\;n\geq p+1,
\end{equation*}
 hold uniformly on $\mathds{R}$. Here the coefficients $\varrho_{\ell}^{(\alpha,s)}\;(\ell=1,2,\ldots)$
 can be determined by the following equation
$$\displaystyle
\frac{\mathrm{e}^{-sz}}{z^\alpha}G_{p,s}(\mathrm{e}^{-z})=1+\sum\limits_{\ell=p}^{\infty}\varrho_{\ell}^{(\alpha,s)}z^\ell,\;|z|<1.
$$
\end{theorem}

Define a generalized fractional
difference operator $\mathcal{J}_{p,s}$ as
$$
\begin{array}{lll}
\mathcal{J}_{p,s}u(x)=\left(1+\varrho_{p}^{(\alpha,s)}h^p\delta_x^p\right)u(x),
\end{array}
$$
where the difference operator $\delta_x^p$ is
defined by
 $$\delta_x^pu(x)=\frac{1}{h^p}\sum_{m=0}^{p}(-1)^m\left(p\atop m
\right)u\left(x+\left(\frac{p}{2}-m\right)h\right),\;\;p\in \mathds{N}.$$

Here, note that the facts
$$
\begin{array}{lll}
\displaystyle\,_{RL}D_{-\infty,x}^{\alpha+p}u(x)
=\frac{\textrm{d}^p}{\textrm{d}x^p}\left(\,_{RL}D_{-\infty,x}^{\alpha}u(x)\right),
\end{array}
$$
$$
\begin{array}{lll}
\displaystyle\,_{RL}D_{x,+\infty}^{\alpha+p}u(x)
=\frac{\textrm{d}^p}{\textrm{d}x^p}\left(\,_{RL}D_{x,+\infty}^{\alpha}u(x)\right),
\end{array}
$$
and
$$
\begin{array}{lll}
\displaystyle\frac{\textrm{d}^p u(x)}{\textrm{d}x^p}=\delta_x^pu(x)+\mathcal{O}(h^p),
\end{array}
$$
then the generalized fractional-compact numerical algorithm formulas for
 Riemann-Liouville (Riesz) derivatives are stated below.

\begin{theorem} Let $u(x)\in C^{[\alpha]+p+1}(\mathds{R})$ and all the derivatives of $u(x)$
up to order $[\alpha]+p+2$ belong to
$L_1(\mathds{R})$. Then there hold
$$
\begin{array}{lll}
\displaystyle\,^{L}\mathcal{B}_{p,s}^{\alpha}u(x)
 =\mathcal{J}_{p,s}\,_{RL}D_{-\infty,x}^{\alpha}u(x)
+\mathcal{O}(h^{p+1}),
\end{array}
$$
and
\begin{equation*}
\displaystyle\,^{R}\mathcal{B}_{p,s}^{\alpha}u(x)
 =\mathcal{J}_{p,s}\,_{RL}D_{x,+\infty}^{\alpha}u(x)
+\mathcal{O}(h^{p+1}),
\end{equation*}
uniformly for $x,s\in\mathds{R}$ and $p\in \mathds{N}$.

Furthermore, one has
\begin{eqnarray}\label{eq49}
\displaystyle \mathcal{J}_{p,s}\frac{\partial^\alpha
u(x)}{\partial{|x|^\alpha}}\left|_{x=x_{j}+sh}\right. =\displaystyle
 -\frac{1}{2\cos\left(\frac{\pi\alpha}{2}\right)}\left[\,^{R}\mathcal{B}_{p,s}^{\alpha}+\,^{L}\mathcal{B}_{p,s}^{\alpha}
\right]u(x_j+sh)\displaystyle+\mathcal {O}(h^{p+1}).\nonumber
\end{eqnarray}
\end{theorem}

{\it{\bf Remark 4:}  The operators $\mathcal{L}$, $\mathcal{\widetilde{L}}$ and $\mathcal{J}_{p,s}$ have the
relations, $\mathcal{J}_{2,-1}=\mathcal{L}$ and $\mathcal{J}_{2,1}=\mathcal{\widetilde{L}}$.}

{\it{\bf Remark 5:}.
It is to be observed that some kinds of $(p+2)$th-order fractional-compact numerical approximation formulas can be obtained by
linear combination of any two different $(p+1)$th-order fractional-compact schemes, where $p\geq2$. Here, we list them as follows.

Suppose $u(x)\in C^{[\alpha]+p+1}(\mathds{R})$ and all the derivatives of $u(x)$ up to order $[\alpha]+p+2$ belong to
$L_1(\mathds{R})$. Define the fractional-compact operator as
$$
\begin{array}{lll}
\displaystyle
\mathcal{{H}}_{p,s_1,s_2}u(x)=\left[\left({\varrho}_{p+1}^{(\alpha,s_2)}-{\varrho}_{p+1}^{(\alpha,s_1)}\right)
+h^p\left({\varrho}_{p}^{(\alpha,s_1)}{\varrho}_{p+1}^{(\alpha,s_2)}
-{{\varrho}}_{p}^{(\alpha,s_2)}{\varrho}_{p+1}^{(\alpha,s_1)}\right)\delta_x^2\right]u(x),
\end{array}
$$
then one has
$$
\begin{array}{lll}
\displaystyle{{\varrho}}_{p+1}^{(\alpha,s_2)}\,^{L}\mathcal{B}_{p,s_1}^{\alpha}u(x)
-{{\varrho}}_{p+1}^{(\alpha,s_1)}\,^{L}\mathcal{{B}}_{p,s_2}^{\alpha}u(x)
 =\mathcal{{H}}_{p,s_1,s_2}\,_{RL}D_{-\infty,x}^{\alpha}u(x)
+\mathcal{O}(h^{p+2}),
\end{array}
$$
and
$$
\begin{array}{lll}
\displaystyle{{\varrho}}_{p+1}^{(\alpha,s_2)}\,^{R}\mathcal{B}_{p,s_1}^{\alpha}u(x)
-{{\varrho}}_{p+1}^{(\alpha,s_1)}\,^{R}\mathcal{{B}}_{p,s_2}^{\alpha}u(x)
 =\mathcal{{H}}_{p,s_1,s_2}\,_{RL}D_{x,+\infty,}^{\alpha}u(x)
+\mathcal{O}(h^{p+2}),
\end{array}
$$
 hold uniformly on $\mathds{R}$.
Furthermore, a kind of ($p+2$)th-order fractional-compact numerical approximation formula for Riesz derivative is given by
\begin{eqnarray*}
\displaystyle \mathcal{H}_{p,s_1,s_2}\frac{\partial^\alpha
u(x)}{\partial{|x|^\alpha}} &=&\displaystyle -\frac{1}{2\cos\left(\frac{\pi\alpha}{2}\right)}\left[{{\varrho}}_{p+1}^{(\alpha,s_2)}\left(
\,^{L}\mathcal{B}_{p,s_1}^{\alpha}+\,^{R}\mathcal{B}_{p,s_1}^{\alpha}\right)\right.\vspace{0.2 cm}\\&&\left.
-{{\varrho}}_{p+1}^{(\alpha,s_1)}\left(
\,^{L}\mathcal{B}_{p,s_2}^{\alpha}+\,^{R}\mathcal{B}_{p,s_2}^{\alpha}\right)
\right]u(x)\displaystyle+\mathcal {O}(h^{p+2}).\nonumber
\end{eqnarray*}

Particularly, if we choose $p = 2$, then the following several commonly four-order fractional-compact schemes can be obtained,
\begin{eqnarray}\label{eq9}
\displaystyle \mathcal{H}_{2,-1,1}\frac{\partial^\alpha
u(x)}{\partial{|x|^\alpha}} &=&\displaystyle -\frac{1}{2\cos\left(\frac{\pi\alpha}{2}\right)}\left[{{\varrho}}_{3}^{(\alpha,1)}\left(
\,^{L}\mathcal{B}_{2,-1}^{\alpha}+\,^{R}\mathcal{B}_{2,-1}^{\alpha}\right)\right.\vspace{0.2 cm}\\&&\left.
-{{\varrho}}_{3}^{(\alpha,-1)}\left(
\,^{L}\mathcal{B}_{2,1}^{\alpha}+\,^{R}\mathcal{B}_{2,1}^{\alpha}\right)
\right]u(x)\displaystyle+\mathcal {O}(h^{4}),\nonumber
\end{eqnarray}
\begin{eqnarray*}
\displaystyle \mathcal{H}_{2,0,1}\frac{\partial^\alpha
u(x)}{\partial{|x|^\alpha}} &=&\displaystyle -\frac{1}{2\cos\left(\frac{\pi\alpha}{2}\right)}\left[{{\varrho}}_{3}^{(\alpha,1)}\left(
\,^{L}\mathcal{B}_{2,0}^{\alpha}+\,^{R}\mathcal{B}_{2,0}^{\alpha}\right)\right.\vspace{0.2 cm}\\&&\left.
-{{\varrho}}_{3}^{(\alpha,0)}\left(
\,^{L}\mathcal{B}_{2,1}^{\alpha}+\,^{R}\mathcal{B}_{2,1}^{\alpha}\right)
\right]u(x)\displaystyle+\mathcal {O}(h^{4}),\nonumber
\end{eqnarray*}
and
\begin{eqnarray*}
\displaystyle \mathcal{H}_{2,0,-1}\frac{\partial^\alpha
u(x)}{\partial{|x|^\alpha}} &=&\displaystyle -\frac{1}{2\cos\left(\frac{\pi\alpha}{2}\right)}\left[{{\varrho}}_{3}^{(\alpha,-1)}\left(
\,^{L}\mathcal{B}_{2,0}^{\alpha}+\,^{R}\mathcal{B}_{2,0}^{\alpha}\right)\right.\vspace{0.2 cm}\\&&\left.
-{{\varrho}}_{3}^{(\alpha,0)}\left(
\,^{L}\mathcal{B}_{2,-1}^{\alpha}+\,^{R}\mathcal{B}_{2,-1}^{\alpha}\right)
\right]u(x)\displaystyle+\mathcal {O}(h^{4}),\nonumber
\end{eqnarray*}
where,
$$
\begin{array}{lll} \displaystyle
\varrho_2^{(\alpha,s)}=-\frac{2\alpha^2+6\alpha s+3s^2}{6\alpha},
\end{array}
$$
$$
\begin{array}{lll} \displaystyle
\varrho_3^{(\alpha,s)}=\frac{3\alpha^3+11\alpha^2 s+12s^2\alpha+4s^3}{12\alpha^2},\;s=-1,0,1.
\end{array}
$$}

\section{Application to Riesz spatial fractional reaction-dispersion equation in one space dimension}\label{sec:4}
 Here, we consider the one-dimension Riesz spatial fractional reaction-dispersion equation in the following form,
\begin{eqnarray}\label{eq10}
  \displaystyle\frac{\partial{{}u(x,t)}}{\partial{t}}
 =-u(x,t)+K_\alpha\frac{\partial^\alpha
u(x,t)}{\partial{|x|^\alpha}}
  +f(x,t),\;1<\alpha<2,\\\hspace{6 cm}\;\;\;0<x<L,\;\;\;0< t\leq T,\nonumber
\end{eqnarray}
  with initial value condition
\begin{equation}\label{eq11}
u(x,0)=u^{0}(x),\;\;0\leq x\leq L,
\end{equation}
and boundary value conditions
\begin{equation}\label{eq12}
u(0,t)=
u(L,t)=0,\;\;0< t\leq T,
\end{equation}
in which parameter $K_\alpha$ is a positive real constant, $u^{0}(x)$ and $f(x,t)$ are given suitably smooth functions.

\subsection{Derivation of the fractional-compact difference scheme}

Let temporal steplength $\tau=\frac{T}{N}$ and spatial steplength $h=\frac{L}{M}$,
where $M$ and $N$ are two positive
integers. Define a partition of $ [0,T]\times[0,L]$ by
$\Omega=\Omega_\tau\times\Omega_h$ with girds
$\Omega_\tau=\{t_k=k\tau\,|\,0\leq k\leq N\}$ and
$\Omega_h=\{x_j=jh\,|\,0\leq j\leq M\}$. For any gird function $u_j^k\in\Omega$, denote
$$
\begin{array}{lll} \displaystyle
 \delta_t u^{k+\frac{1}{2}}_j=\frac{u^{k+1}_j-u^{k}_j}{\tau},\;\; u^{k+\frac{1}{2}}_j=\frac{u^{k+1}_j+u^{k}_j}{2}.
\end{array}
$$
And set
$$
\begin{array}{lll}
\displaystyle \delta_x^\alpha = -\frac{1}{2\cos\left(\frac{\pi\alpha}{2}\right)}\left(
\,^{L}\mathcal{A}_{2}^{\alpha}+\,^{R}\mathcal{A}_{2}^{\alpha}\right),\;\;1<\alpha<2.
\end{array}
$$
Now we consider equation (\ref{eq10}) at the point $(x_j, t)$. Then we have
\begin{equation}\label{eq13}
 \frac{\partial{{}u(x_j,t)}}{\partial{t}}
 =-u(x_j,t)+K_\alpha\frac{\partial^\alpha
u(x_j,t)}{\partial{|x|^\alpha}}
  +f(x_j,t),\;\;\;0\leq j\leq M.
\end{equation}
Operating operator $\mathcal{L}$ on both sides of (\ref{eq13}) yields
$$
\begin{array}{lll} \displaystyle
 \mathcal{L}\frac{\partial{{}u(x_j,t)}}{\partial{t}}
 =-\mathcal{L}u(x_j,t)+K_\alpha\mathcal{L}\frac{\partial^\alpha
u(x_j,t)}{\partial{|x|^\alpha}}
  +\mathcal{L}f(x_j,t),\;\;\;0\leq j\leq M.
\end{array}
$$
Noticing the definition of operator $\mathcal{L}$ and equation (\ref{eq7}) gives
$$
\begin{array}{lll} \displaystyle
 \mathcal{L}\frac{\partial{{}u(x_j,t)}}{\partial{t}}
 =-\mathcal{L}u(x_j,t)+K_\alpha\delta_x^\alpha
u(x_j,t)
  +\mathcal{L}f(x_j,t)+\mathcal {O}(h^3),\;\;\;0\leq j\leq M.
\end{array}
$$
Taking $t=t_{k+\frac{1}{2}}$ and using the Taylor expansion, one has
\begin{eqnarray}\label{eq14}
 \mathcal{L}\delta_t u(x_j,t_{k+\frac{1}{2}})
 =-\mathcal{L}u(x_j,t_{k+\frac{1}{2}})+K_\alpha\delta_x^\alpha
u(x_j,t_{k+\frac{1}{2}})\vspace{0.2 cm}\displaystyle\\
  +\mathcal{L}f(x_j,t_{k+\frac{1}{2}})+R_j^k,\;\;\;0\leq j\leq M.\nonumber
\end{eqnarray}
where there exists a positive constant $C_1$ such that
$$
\begin{array}{lll} \displaystyle
|R_j^k|\leq C_1(\tau^2+h^3),\;0\leq k\leq N-1,\;1\leq j\leq M-1.
\end{array}
$$

Omitting the high-order terms $R_j^k$ of (\ref{eq14}) and letting $u_j^k$ be the numerical approximation
of function $u(x_j,t_k)$, then we can obtain the following fractional-compact difference scheme
for equation (\ref{eq10}), together with initial and boundary value conditions (\ref{eq11}) and (\ref{eq12}) as follows,
\begin{eqnarray}\label{eq15}
 \mathcal{L}\delta_t u_j^{k+\frac{1}{2}}
 =-\mathcal{L}u_j^{k+\frac{1}{2}}+K_\alpha\delta_x^\alpha
u_j^{k+\frac{1}{2}}
  +\mathcal{L}f_j^{k+\frac{1}{2}},\\\;\;\;0\leq k\leq N-1,\;\;1\leq j\leq M-1,\nonumber
\end{eqnarray}
\begin{equation}\label{eq16}
u_j^0=u^{0}(x_j),\;\;0\leq j\leq M,
\end{equation}
\begin{equation}\label{eq17}
u_0^k=
u_M^k=0,\;\;1\leq k\leq N.
\end{equation}

\subsection{Analysis of the fractional-compact difference scheme}
Let
$$
\begin{array}{lll} \displaystyle
\mathcal{V}_h=\{\bm{v}|\bm{v}=(v_0,v_1,\ldots,v_M),v_0=v_M=0\}
\end{array}
$$
be the space of grid
functions. Then for any $\bm{u},\bm{v}\in\mathcal{V}_h$, we can define the discrete inner products below,
$$
\begin{array}{lll} \displaystyle
(\bm{u},\bm{v})=h\sum_{j=1}^{M-1}u_jv_j,\;\;(\delta_x\bm{u},\delta_x\bm{v})=h\sum_{j=1}^{M}\left(\delta_x u_{j-\frac{1}{2}}\right)\left(\delta_x v_{j-\frac{1}{2}}\right),
\end{array}
$$
and associated norms below,
$$
\begin{array}{lll} \displaystyle
||\bm{u}||^2=(\bm{u},\bm{u}),\;\;||\delta_x\bm{u}||^2=(\delta_x\bm{u},\delta_x\bm{u}).
\end{array}
$$

Next,
we list several lemmas which will be used later on.

\begin{lemma}\label{Le4.1}
Operator $\mathcal{L}$ is self-adjoint, that is, for any $\bm{u},\bm{v}\in \mathcal{V}_h$, there holds,
$$
\begin{array}{lll} \displaystyle
(\mathcal{L}\bm{u},\bm{v})=(\bm{u},\mathcal{L}\bm{v}).
\end{array}
$$
\end{lemma}

\begin{proof}
It follows from the definition of operator $\mathcal{L}$ that
$$
\begin{array}{lll} \displaystyle
\left(\mathcal{L}\bm{u},\bm{v}\right)&=&\displaystyle
\left(\left(1+\sigma_{2}^{(\alpha)}h^2\delta_x^2\right)\bm{u},\bm{v}\right)=
\left(\bm{u},\bm{v}\right)-\sigma_{2}^{(\alpha)}h^2\left(\delta_x\bm{u},\delta_x\bm{v}\right)\vspace{0.2 cm}\\&=&\displaystyle
\left(\bm{u},\bm{v}\right)+\sigma_{2}^{(\alpha)}h^2\left(\bm{u},\delta_x^2\bm{v}\right)
=\left(\bm{u},\left(1+\sigma_{2}^{(\alpha)}h^2\delta_x^2\right)\bm{v}\right)
\vspace{0.2 cm}\\&=&\displaystyle\left(\bm{u},\mathcal{L}\bm{v}\right).
\end{array}
$$
All this ends the proof.
\end{proof}

\begin{lemma}\label{Le4.2}
 For any $\bm{u}\in\mathcal{V}_h$, there holds that
$$
\begin{array}{lll} \displaystyle
\left(\frac{4\sqrt{6}}{3}-3\right)||\bm{u}||^2\leq(\mathcal{L}\bm{u},\bm{u})\leq||\bm{u}||^2,\;\;1<\alpha<2.
\end{array}
$$
\end{lemma}
\begin{proof}
On one hand, note that
$\sigma_{2}^{(\alpha)}=
-\frac{2\alpha^2-6\alpha+3}{6\alpha}\in\left(\frac{1}{12},1-\frac{\sqrt{6}}{3}\right]$ for $1<\alpha<2$, then one has
$$
\begin{array}{lll} \displaystyle
\left(\mathcal{L}\bm{u},\bm{u}\right)=
\left(\bm{u},\bm{v}\right)-\sigma_{2}^{(\alpha)}h^2\left(\delta_x\bm{u},\delta_x\bm{v}\right)
=||\bm{u}||^2-\sigma_{2}^{(\alpha)}h^2||\delta_x\bm{u}||^2\leq||\bm{u}||^2.
\end{array}
$$

On the other hand, using the inverse estimate $||\delta_x\bm{u}||^2\leq\frac{4}{h^2}||\bm{u}||^2$, we reach that
$$
\begin{array}{lll} \displaystyle
\left(\mathcal{L}\bm{u},\bm{u}\right)
=||\bm{u}||^2-\sigma_{2}^{(\alpha)}h^2||\delta_x\bm{u}||^2\geq\left(\frac{4\sqrt{6}}{3}-3\right)||\bm{u}||^2.
\end{array}
$$
This finishes the proof.
\end{proof}

\begin{lemma}\cite{Ding4}\label{Le4.3}
 For any $\bm{u}\in\mathcal{V}_h$, there holds that
$$
\begin{array}{lll} \displaystyle
(\delta_x^\alpha\bm{u},\bm{u})\leq0.
\end{array}
$$
\end{lemma}

\begin{lemma}(Grownall's inequality \cite{Quarteroni})\label{Le4.4}  Assume that $\{k_n\}$ and  $\{p_n\}$ are nonnegative sequences, and
the sequence $\{\phi_n\}$ satisfies
$$
\begin{array}{lll}
\displaystyle
\phi_0\leq q_0,\;\;\phi_n\leq q_0+\sum_{\ell=0}^{n-1}p_{\ell}+\sum_{\ell=0}^{n-1}k_{\ell}\phi_{\ell},\;n\geq1,
\end{array}
$$
where $q_{0}\geq0$. Then the sequence $\{\phi_n\}$
$$
\begin{array}{lll}
\displaystyle
\phi_n\leq \left(q_0+\sum_{\ell=0}^{n-1}p_{\ell}\right)\exp\left(\sum_{\ell=0}^{n-1}k_{\ell}\right)
\end{array}
$$
holds for $n\geq 1$.
\end{lemma}

Now, we turn to study the stability of finite difference scheme (\ref{eq15}) with
(\ref{eq16}) and (\ref{eq17}).
\begin{theorem}
The difference scheme (\ref{eq15}) with
(\ref{eq16}) and (\ref{eq17}) is unconditionally stable with respect to the initial
values.
\end{theorem}
\begin{proof}
Let $$\xi_j^k=u_j^k-v_j^k,\;0\leq k\leq N,\;0\leq j\leq M,$$
where $u_j^k$ and $v_j^k$ are the solutions of the following two equations, respectively,
\begin{eqnarray}\label{eq18}
 \mathcal{L}\delta_t u_j^{k+\frac{1}{2}}
 =-\mathcal{L}u_j^{k+\frac{1}{2}}+K_\alpha\delta_x^\alpha
u_j^{k+\frac{1}{2}}
  +\mathcal{L}f_j^{k+\frac{1}{2}},\\\;\;\;0\leq k\leq N-1,\;\;1\leq j\leq M-1,\nonumber
\end{eqnarray}
\begin{equation}\label{eq19}
u_j^0=u^{0}(x_j),\;\;0\leq j\leq M,
\end{equation}
\begin{equation}\label{eq20}
u_0^k=
u_M^k=0,\;\;1\leq k\leq N.
\end{equation}
and
\begin{eqnarray}\label{eq21}
 \mathcal{L}\delta_t v_j^{k+\frac{1}{2}}
 =-\mathcal{L}v_j^{k+\frac{1}{2}}+K_\alpha\delta_x^\alpha
v_j^{k+\frac{1}{2}}
  +\mathcal{L}f_j^{k+\frac{1}{2}},\\\;\;\;0\leq k\leq N-1,\;\;1\leq j\leq M-1,\nonumber
\end{eqnarray}
\begin{eqnarray}\label{eq22}
v_j^0=u^{0}(x_j)+\varepsilon_j,\;\;0\leq j\leq M,
\end{eqnarray}
\begin{eqnarray}\label{eq23}
v_0^k=
v_M^k=0,\;\;1\leq k\leq N.
\end{eqnarray}
Subtracting (\ref{eq18})-(\ref{eq20}) from (\ref{eq21})-(\ref{eq23}) gives the perturbation equations as follows,
\begin{eqnarray}\label{eq24}
 \mathcal{L}\delta_t \xi_j^{k+\frac{1}{2}}
 =-\mathcal{L}\xi_j^{k+\frac{1}{2}}+K_\alpha\delta_x^\alpha
\xi_j^{k+\frac{1}{2}},\\\;\;\;0\leq k\leq N-1,\;\;1\leq j\leq M-1,\nonumber
\end{eqnarray}
$$
\begin{array}{lll} \displaystyle
\xi_j^0=-\varepsilon_j,\;\;0\leq j\leq M,
\end{array}
$$
$$
\begin{array}{lll} \displaystyle
\xi_0^k=
\xi_M^k=0,\;\;1\leq k\leq N.
\end{array}
$$
Taking the inner product of (\ref{eq24}) with $ \bm{\xi}^{k+\frac{1}{2}}$, replacing $k$ by $n$, and summing from $n =0$
to $k-1$ yield
\begin{eqnarray}\label{eq25}
 &&\displaystyle\sum_{n=0}^{k-1}\left(\mathcal{L}\delta_t \bm{\xi}^{n+\frac{1}{2}},\bm{\xi}^{n+\frac{1}{2}}\right)
 + \sum_{n=0}^{k-1}\left(\mathcal{L}\bm{\xi}^{n+\frac{1}{2}},\bm{\xi}^{n+\frac{1}{2}}\right)\\&&\displaystyle=K_\alpha \sum_{n=0}^{k-1}
\left(\delta_x^\alpha\bm{\xi}^{n+\frac{1}{2}},\bm{\xi}^{n+\frac{1}{2}}\right),\;\;\;0\leq k\leq N-1,\;\;1\leq j\leq M-1.\nonumber
\end{eqnarray}
For the first term on the left hand side of (\ref{eq25}), using Lemma \ref{Le4.1} leads to
$$
\begin{array}{lll} \displaystyle
 \sum_{n=0}^{k-1} \left(\mathcal{L}\delta_t \bm{\xi}^{n+\frac{1}{2}},\bm{\xi}^{n+\frac{1}{2}}\right)=
 \frac{1}{2\tau} \sum_{n=0}^{k-1}\left[\left(\mathcal{L} \bm{\xi}^{n+1},\bm{\xi}^{n+1}\right)
 -\left(\mathcal{L} \bm{\xi}^{n},\bm{\xi}^{n}\right)\right].
\end{array}
$$
For the second term on the left hand side of (\ref{eq25}), in view of Lemma \ref{Le4.2}, we have
$$
\begin{array}{lll} \displaystyle
\sum_{n=0}^{k-1}\left(\mathcal{L}\bm{\xi}^{n+\frac{1}{2}},\bm{\xi}^{n+\frac{1}{2}}\right)\geq0.
\end{array}
$$
For the term on the right hand side of (\ref{eq25}), it follows from Lemma \ref{Le4.3} that
$$
\begin{array}{lll} \displaystyle
K_\alpha \sum_{n=0}^{k-1}
\left(\delta_x^\alpha\bm{\xi}^{n+\frac{1}{2}},\bm{\xi}^{n+\frac{1}{2}}\right)\leq0.
\end{array}
$$

Hence, combining the above discussion, one has
$$
\begin{array}{lll} \displaystyle
 \sum_{n=0}^{k-1}\left[\left(\mathcal{L} \bm{\xi}^{n+1},\bm{\xi}^{n+1}\right)
 -\left(\mathcal{L} \bm{\xi}^{n},\bm{\xi}^{n}\right)\right]\leq0,
\end{array}
$$
i.e.,
$$
\begin{array}{lll} \displaystyle
\left(\mathcal{L} \bm{\xi}^{k},\bm{\xi}^{k}\right)
 \leq\left(\mathcal{L} \bm{\xi}^{0},\bm{\xi}^{0}\right).
\end{array}
$$
 Using Lemma \ref{Le4.2} again leads to
$$
\begin{array}{lll} \displaystyle
||\bm{\xi}^{k}||\leq\frac{\sqrt{5(4\sqrt{6}+9)}}{5}
||\bm{\xi}^{0}||=\frac{\sqrt{5(4\sqrt{6}+9)}}{5}
||\bm{\varepsilon}||.
\end{array}
$$
The proof is thus completed.
\end{proof}

Finally, we give the convergence result as follows.

\begin{theorem}
 Let $u_j^k$ and $u(x, t)$ be the solutions of the finite difference
 scheme (\ref{eq15})-(\ref{eq17}) and problem (\ref{eq10})-(\ref{eq12}), respectively.
 Denote $e_j^k=u_j^k-u(x_j, t_k)$, $\;0\leq k\leq N,\;\;0\leq j\leq M$.
Then there holds
$$
\begin{array}{lll} \displaystyle
||\bm{e}^{k}||\leq \frac{\sqrt{6LT}}{6}C_1\exp\left(
\frac{1}{3}T\right)(\tau^2+h^3),
\end{array}
$$
where $C_1$ is a constant independent of $\tau$ and $h$.
\end{theorem}

\begin{proof}
Subtracting (\ref{eq10})-(\ref{eq12}) from (\ref{eq15})-(\ref{eq17}), we get the following error equation
\begin{eqnarray}\label{eq26}
 \mathcal{L}\delta_t e_j^{k+\frac{1}{2}}
 =-\mathcal{L}e_j^{k+\frac{1}{2}}+K_\alpha\delta_x^\alpha
e_j^{k+\frac{1}{2}}
  +R_j^{k+1},\\\;\;\;0\leq k\leq N-1,\;\;1\leq j\leq M-1,\nonumber
\end{eqnarray}
$$
\begin{array}{lll} \displaystyle
e_j^0=0,\;\;0\leq j\leq M,
\end{array}
$$
$$
\begin{array}{lll} \displaystyle
e_0^k=
e_M^k=0,\;\;1\leq k\leq N.
\end{array}
$$
Taking the inner product of (\ref{eq26}) with $ \bm{e}^{k+\frac{1}{2}}$, replacing $k$ by $n$, and summing up from $n =0$
to $k-1$, lead to
\begin{eqnarray}\label{eq27}
 \sum_{n=0}^{k-1}\left(\mathcal{L}\delta_t \bm{e}^{n+\frac{1}{2}},\bm{e}^{n+\frac{1}{2}}\right)
 + \sum_{n=0}^{k-1}\left(\mathcal{L}\bm{e}^{n+\frac{1}{2}},\bm{e}^{n+\frac{1}{2}}\right)\\=K_\alpha \sum_{n=0}^{k-1}
\left(\delta_x^\alpha\bm{e}^{n+\frac{1}{2}},\bm{e}^{n+\frac{1}{2}}\right)+
\sum_{n=0}^{k-1}\left(\bm{R}^{n+1},\bm{e}^{n+\frac{1}{2}}\right).\nonumber
\end{eqnarray}
For the last term on the right hand side of (\ref{eq27}), we have the following estimate,
$$
\begin{array}{lll} \displaystyle
\sum_{n=0}^{k-1}\left(\bm{R}^{n+1},\bm{e}^{n+\frac{1}{2}}\right)\leq
\frac{1}{4}||\bm{e}^{k}||^2
+\frac{1}{2}\sum_{n=0}^{k-1}||\bm{e}^{n}||^2
+\frac{1}{2}\sum_{n=0}^{k-1}||\bm{R}^{n+1}||^2.
\end{array}
$$

Hence, the following result can be obtained by combining with the above discussion,
$$
\begin{array}{lll} \displaystyle
\left(\mathcal{L} \bm{e}^{k},\bm{e}^{k}\right)
 \leq\left(\mathcal{L} \bm{e}^{0},\bm{e}^{0}\right)+\frac{1}{4}||\bm{e}^{k}||^2
+\frac{1}{2}\sum_{n=0}^{k-1}||\bm{e}^{n}||^2
+\frac{1}{2}\sum_{n=0}^{k-1}||\bm{R}^{n+1}||^2.
\end{array}
$$
Noticing
$$
\begin{array}{lll} \displaystyle
\left\|{\bm{R}}^{n+1}\right\|^2=h\sum_{j=1}^{M-1}\left(R_j^{n+1}\right)^2
\leq C_1^2L(\tau^2+h^3)^2
\end{array}
$$
and using Lemma \ref{Le4.2} give
\begin{eqnarray*}
||\bm{e}^{k}||^2\leq
\frac{2(16\sqrt{6}+39)}{1019}\tau\sum_{n=0}^{k-1}||\bm{e}^{n}||^2
+\frac{2(16\sqrt{6}+39)}{1019}C_1^2LT(\tau^2+h^3)^2.
\end{eqnarray*}

Finally, it is easy to obtain the following result by using Lemma \ref{Le4.3},
\begin{eqnarray*}
||\bm{e}^{k}||^2&\leq& \displaystyle
\frac{2(16\sqrt{6}+39)}{1019}\exp\left(
\frac{2(16\sqrt{6}+39)}{1019}T\right)C_1^2LT(\tau^2+h^3)^2\vspace{0.2 cm}\\
&\leq& \displaystyle
\frac{1}{6}\exp\left(
\frac{1}{6}T\right)C_1^2LT(\tau^2+h^3)^2,
\end{eqnarray*}
i.e.,
$$
\begin{array}{lll} \displaystyle
||\bm{e}^{k}||\leq
\frac{\sqrt{6LT}}{6}C_1\exp\left(
\frac{1}{12}T\right)(\tau^2+h^3).
\end{array}
$$
Therefore, the proof is finished.
\end{proof}

\section{Application to the two-dimensional equation}

In this section, we consider the following two-dimension
 Riesz spatial fractional reaction-dispersion equation,
\begin{eqnarray}\label{eq28}
 \frac{\partial{{}u(x,y,t)}}{\partial{t}}
 =-u(x,y,t)+K_\alpha\frac{\partial^\alpha
u(x,y,t)}{\partial{|x|^\alpha}}+K_\beta\frac{\partial^\alpha
u(x,y,t)}{\partial{|y|^\beta}}
 \\ +f(x,y,t),  (x,y;t)\in\Omega\times(0,T], \nonumber
  \end{eqnarray}
  with the initial value condition
  \begin{eqnarray}\label{eq29}
u(x,y,0)=u^0(x,y),\;\;(x,y)\in{\overline{\Omega}},
\end{eqnarray}
and the boundary value conditions
\begin{eqnarray}\label{eq30}
u(x,y,t)=0,\;(x,y;t)\in\partial\Omega\times(0,T],
\end{eqnarray}
where $\alpha,\beta\in(1,2)$, $\Omega=(0,L_a)\times(0,L_b)$, coefficients $K_\alpha$ and $K_\beta$ are two positive constants,
$f(x,y,t)$ and $u^{0}(x,y)$
are suitably smooth.

\subsection{Derivation of the fractional-compact difference scheme}
For a given positive integer $N$, denote the timestep size $\tau=\frac{N}{T}$ and grid points $t_k=k\tau$, $0\leq k\leq N$.
 Set the space stepsizes $h_a=\frac{L_a}{M_a}$ and $h_b=\frac{L_b}{M_b}$, where $M_a$ and $M_b$
are two positive integers. And the according grid points are
$x_i=ih_a$, $0\leq i\leq M_a$, and $y_j=jh_b$, $0\leq j\leq M_b$.
 In addition, let ${\overline{\Omega}}_h=\{(x_i,y_j|0\leq i\leq M_a,0\leq j\leq M_b\}$, $\Omega_h=\overline{\Omega}_h\cap\Omega$,
 and $\partial\Omega_h=\overline{\Omega}_h\cap\partial\Omega$.

Define the fractional-compact difference operators $\mathcal{L}_x$ and $\mathcal{L}_y$ as
$$
\begin{array}{lll}
\mathcal{L}_xu(x,y,t)=\left(1+\sigma_{2}^{(\alpha)}h_a^2\delta_x^2\right)u(x,y,t),
\end{array}
$$
and
$$
\begin{array}{lll}
\mathcal{L}_yu(x,y,t)=\left(1+\sigma_{2}^{(\beta)}h_b^2\delta_y^2\right)u(x,y,t).
\end{array}
$$
For brevity, set
$$
\begin{array}{lll}
\displaystyle
\,^{L}\mathcal{A}_{2,x}^{\alpha}u(x,y,t)
=\frac{1}{h_a^{\alpha}}\sum\limits_{\ell=0}^{\left[\frac{x}{h_a}\right]+1}
\kappa_{2,\ell}^{(\alpha)}u\left(x-(\ell-1)h_a,y,t\right),
\end{array}
$$
$$
\begin{array}{lll}
\displaystyle
\,^{L}\mathcal{A}_{2,y}^{\beta}u(x,y,t)
=\frac{1}{h_b^{\beta}}\sum\limits_{\ell=0}^{\left[\frac{y}{h_b}\right]+1}
\kappa_{2,\ell}^{(\beta)}u\left(x,y-(\ell-1)h_b,t\right),
\end{array}
$$
$$
\begin{array}{lll}
\displaystyle
\,^{R}\mathcal{A}_{2,x}^{\alpha}u(x,y,t)
=\frac{1}{h_a^{\alpha}}\sum\limits_{\ell=0}^{\left[\frac{L_a-x}{h_a}\right]+1}
\kappa_{2,\ell}^{(\alpha)}u\left(x+(\ell-1)h_a,y,t\right),
\end{array}
$$
and
$$
\begin{array}{lll}
\displaystyle
\,^{R}\mathcal{A}_{2,y}^{\beta}u(x,y,t)
=\frac{1}{h_b^{\beta}}\sum\limits_{\ell=0}^{\left[\frac{L_b-y}{h_b}\right]+1}
\kappa_{2,\ell}^{(\beta)}u\left(x,y+(\ell-1)h_b,t\right).
\end{array}
$$
Then one has
\begin{eqnarray}\label{eq31}
\displaystyle \mathcal{L}_x\frac{\partial^\alpha
u(x,y,t)}{\partial{|x|^\alpha}}= -\frac{1}{2\cos\left(\frac{\pi\alpha}{2}\right)}\left(
\,^{L}\mathcal{A}_{2,x}^{\alpha}+\,^{R}\mathcal{A}_{2,x}^{\alpha}\right)u(x,y,t)\\+\mathcal{O}(h_a^3),\nonumber
\end{eqnarray}
and
\begin{eqnarray}\label{eq32}
\displaystyle \mathcal{L}_y\frac{\partial^\beta
u(x,y,t)}{\partial{|y|^\beta}}= -\frac{1}{2\cos\left(\frac{\pi\beta}{2}\right)}\left(
\,^{L}\mathcal{A}_{2,y}^{\beta}+\,^{R}\mathcal{A}_{2,y}^{\beta}\right)u(x,y,t)\\+\mathcal{O}(h_b^3).\nonumber
\end{eqnarray}

In order to simplify the expressions, let
$$
\begin{array}{lll}
\displaystyle \delta_x^\alpha = -\frac{1}{2\cos\left(\frac{\pi\alpha}{2}\right)}\left(
\,^{L}\mathcal{A}_{2,x}^{\alpha}+\,^{R}\mathcal{A}_{2,x}^{\alpha}\right)
\end{array}
$$
and
$$
\begin{array}{lll}
\displaystyle \delta_y^\beta = -\frac{1}{2\cos\left(\frac{\pi\beta}{2}\right)}\left(
\,^{L}\mathcal{A}_{2,y}^{\beta}+\,^{R}\mathcal{A}_{2,y}^{\beta}\right).
\end{array}
$$
Accordingly, equations (\ref{eq31}) and (\ref{eq32}) can be rewritten as,
\begin{eqnarray}\label{eq33}
\displaystyle \mathcal{L}_x\frac{\partial^\alpha
u(x,y,t)}{\partial{|x|^\alpha}}= \delta_x^\alpha u(x,y,t)+\mathcal{O}(h_a^3),
\end{eqnarray}
and
\begin{eqnarray}\label{eq34}
\displaystyle \mathcal{L}_y\frac{\partial^\beta
u(x,y,t)}{\partial{|y|^\beta}}= \delta_y^\beta u(x,y,t)+\mathcal{O}(h_b^3).
\end{eqnarray}

Similar to the one-dimension case, utilizing the central difference scheme in time direction and
 fractional-compact difference formulas (\ref{eq33}) and (\ref{eq34}) in space directions, we get
\begin{eqnarray}\label{eq35}
 &&\mathcal{L}_x\mathcal{L}_y\delta_t u(x_i,y_j,t_{k+\frac{1}{2}})
 =-\mathcal{L}_x\mathcal{L}_yu(x_i,y_j,t_{k+\frac{1}{2}})\vspace{0.2 cm}\\&&+K_\alpha\mathcal{L}_y\delta_x^\alpha
u(x_i,y_j,t_{k+\frac{1}{2}}) +K_\beta\mathcal{L}_x\delta_y^\beta
u(x_i,y_j,t_{k+\frac{1}{2}})\nonumber\vspace{0.2 cm}\\&&
  +\mathcal{L}_x\mathcal{L}_yf(x_i,y_j,t_{k+\frac{1}{2}})+R_{i,j}^k,\;\;1\leq i\leq M_x-1,\;\nonumber
  \vspace{0.2 cm}\\&&1\leq j\leq M_y-1,\;0\leq k\leq N-1,\nonumber
\end{eqnarray}
where there exists a positive constant $C_2$ such that
$$
\begin{array}{lll} \displaystyle
|R_{i,j}^k|\leq C_2(\tau^2+h_x^3+h_y^3),\;1\leq i\leq M_x-1,\;1\leq j\leq M_y-1,\;0\leq k\leq N-1.
\end{array}
$$

Omitting the high-order term $R_{i,j}^k$
 and replacing the function $u(x_i,y_j,t_{k+\frac{1}{2}})$ by its numerical
  approximation $u_{i,j}^{k+\frac{1}{2}}$ in (\ref{eq33}), then a finite difference scheme
for equations (\ref{eq28})-(\ref{eq30}) is obtained,
\begin{eqnarray}\label{eq36}
 &&\mathcal{L}_x\mathcal{L}_y\delta_t u_{i,j}^{k+\frac{1}{2}}
 =-\mathcal{L}_x\mathcal{L}_yu_{i,j}^{k+\frac{1}{2}}+K_\alpha\mathcal{L}_y\delta_x^\alpha
u_{i,j}^{k+\frac{1}{2}} +K_\beta\mathcal{L}_x\delta_y^\beta
u_{i,j}^{k+\frac{1}{2}}\vspace{0.2 cm}\\
 && +\mathcal{L}_x\mathcal{L}_yf_{i,j}^{k+\frac{1}{2}},1\leq i\leq M_x-1,\;1\leq j\leq M_y-1,\;0\leq k\leq N-1,\nonumber
  \end{eqnarray}
  \begin{eqnarray}\label{eq37}
u_{i,j}^{0}=u^0(x_i,y_j),\;\;(x_i,y_j)\in{\overline{\Omega}}_h,
\end{eqnarray}
\begin{eqnarray}\label{eq38}
u_{i,j}^{k}=0,\;(x_i,y_j)\in\partial{\Omega}_h,\;1\leq k\leq N.
\end{eqnarray}

\subsection{Analysis of the fractional-compact difference scheme}
Almost the same as the one dimensional case,
let
$$
\begin{array}{lll} \displaystyle
\mathcal{V}_{h_{a},h_{b}}=\{\bm{v}|\bm{v}=\{v_{i,j}\}\;\textrm{is\,\,a\,grid\,\,function\,\,on}\,\,{\Omega}_h\,\,\textrm{and}\,\,v_{i,j} = 0\,\,\textrm{if}\,\,(x_i,y_j)\in\partial{\Omega}_h\},
\end{array}
$$
then for any $\bm{u},\bm{v}\in\mathcal{V}_{h_{a},h_{b}}$,
we introduce the discrete inner products and corresponding norms below,
$$
\begin{array}{lll} \displaystyle
(\bm{u},\bm{v})=h_ah_b\sum_{i=1}^{M_a-1}\sum_{j=1}^{M_b-1}u_{i,j}v_{i,j},
\end{array}
$$
$$
\begin{array}{lll} \displaystyle
(\delta_x\bm{u},\delta_x\bm{v})=h_ah_b\sum_{i=1}^{M_a}\sum_{j=1}^{M_b-1}
\left(\delta_x u_{i-\frac{1}{2},j}\right)\left(\delta_x v_{i-\frac{1}{2},j}\right),
\end{array}
$$
and
$$
\begin{array}{lll} \displaystyle
||\bm{u}||^2=(\bm{u},\bm{u}),\;\;||\delta_x\bm{u}||^2=(\delta_x\bm{u},\delta_x\bm{u}).
\end{array}
$$

The following definition and lemmas are useful for our discussion \cite{Laub}.

\begin{definition}
If $\mathbf{A}=(a_{ij})$ is an $m \times n$ matrix and
$\mathbf{B}=(b_{ij})$ is a $p \times q$ matrix, then the Kronecker product $\mathbf{A}\otimes \mathbf{B}$ is an $mp \times nq$ block matrix
and is denoted by
$$\mathbf{A}\otimes \mathbf{B}=
\left(
  \begin{array}{cccc}
   a_{11}\mathbf{B} & a_{12}\mathbf{B}&  \cdots & a_{1n}\mathbf{B} \vspace{0.2 cm}\\
     a_{21}\mathbf{B} &   a_{22}\mathbf{B} & \cdots&  a_{2n}\mathbf{B} \vspace{0.2 cm}\\
    \vdots &\vdots  & \ddots & \vdots \vspace{0.2 cm}\\
 a_{m1}\mathbf{B}& a_{m2}\mathbf{B} &  \cdots& a_{mn}\mathbf{B} \vspace{0.2 cm}\\
  \end{array}
\right)_.
$$
\end{definition}

\begin{lemma}\label{Le5.1}
Assume that $\mathbf{A}\in \mathds{R}^{n\times n}$ has eigenvalues $\{\lambda_j\}^n_{j=1}$, and that $\mathbf{B}\in \mathds{R}^{m\times m}$
has eigenvalues $\{\mu_j\}^m_{j=1}$.
Then the $mn$ eigenvalues of $\mathbf{A}\otimes \mathbf{B}$ are:
$$
\begin{array}{ll}
\displaystyle
 \lambda_1\mu_1,\ldots\lambda_1\mu_m; \lambda_2\mu_1,\ldots\lambda_2\mu_m;\ldots; \lambda_n\mu_1,\ldots\lambda_n\mu_m.
\end{array}
$$
\end{lemma}

\begin{lemma}\label{Le5.2}
Let $\mathbf{A}\in \mathds{R}^{m\times n}$, $\mathbf{B}\in \mathds{R}^{r\times s}$,
 $\mathbf{C}\in \mathds{R}^{n\times p}$, $\mathbf{D}\in \mathds{R}^{s\times t}$. Then
$$
\begin{array}{ll}
\displaystyle
 \left(\mathbf{A}\otimes \mathbf{B}\right)\left(\mathbf{C}\otimes \mathbf{D}\right)=\mathbf{AC}\otimes \mathbf{BD}.
\end{array}
$$
Moreover, if $\mathbf{A}\in \mathds{R}^{n\times n}$, $\mathbf{B}\in \mathds{R}^{m\times m}$, $\mathbf{I}_n$ and $\mathbf{I}_m$
are unit matrices of order $n$, $m$, respectively, then matrices $\mathbf{I}_m\otimes \mathbf{A}$
 and $\mathbf{B}\otimes\mathbf{I}_n$ can commute with each other.
\end{lemma}

\begin{lemma}\label{Le5.3}
For all $\mathbf{A}$ and $\mathbf{B}$, there holds
$$
\begin{array}{ll}
\displaystyle
\left(\mathbf{A}\otimes \mathbf{B}\right)^{T}=\mathbf{A}^{T}\otimes \mathbf{B}^{T}.
\end{array}
$$
\end{lemma}

\begin{lemma}\cite{Ding4}\label{Le5.4} Denote
$$\mathbf{E}_p^{(\gamma)}=
\left(
  \begin{array}{ccccc}
    \kappa_{2,1}^{(\gamma)} & \kappa_{2,0}^{(\gamma)} & 0 & \cdots & 0 \vspace{0.2 cm}\\
    \kappa_{2,2}^{(\gamma)} & \kappa_{2,1}^{(\gamma)} & \kappa_{2,0}^{(\gamma)} & 0& \cdots \vspace{0.2 cm}\\
    \vdots &\vdots & \ddots & \ddots & \ddots \vspace{0.2 cm}\\
  \kappa_{2,M_p-2}^{(\gamma)} & \ldots & \kappa_{2,2}^{(\gamma)} &\kappa_{2,1}^{(\gamma)}& \kappa_{2,0}^{(\gamma)} \vspace{0.2 cm}\\
  \kappa_{2,M_p-1}^{(\gamma)}& \kappa_{2,M_p-2}^{(\gamma)}& \ldots & \kappa_{2,2}^{(\gamma)}&  \kappa_{2,1}^{(\gamma)} \vspace{0.2 cm}\\
  \end{array}
\right)_{(M_p-1)\times(M_p-1).}
$$
Then matrix $\left(\mathbf{E}_p^{(\gamma)}+\mathbf{E}_p^{(\gamma)^{T}}\right)$ is semi-negative definite.
\end{lemma}

\begin{lemma}\label{Le5.5}
For any mesh functions $\bm{u},\bm{v}\in \mathcal{V}_{h_{a},h_{b}}$, there exists a symmetric positive
 definite operator $\mathcal{P}$ such that
$$
\begin{array}{ll}
\displaystyle
((K_\alpha\mathcal{L}_y
\delta_x^\alpha+K_\beta\mathcal{L}_x\delta_y^\beta)\bm{u},\bm{v})=-(\mathcal{P}\bm{u},\mathcal{P}\bm{v}),\;K_\alpha,K_\beta>0.
\end{array}
$$
\end{lemma}

\begin{proof}
Firstly, we rewrite the inner product $
((K_\alpha\mathcal{L}_y
\delta_x^\alpha+K_\beta\mathcal{L}_x\delta_y^\beta)\bm{u},\bm{v})
$ in a matrix form,
$$
\begin{array}{ll}
\displaystyle
((K_\alpha\mathcal{L}_y
\delta_x^\alpha+K_\beta\mathcal{L}_x\delta_y^\beta)\bm{u},\bm{v})=h_ah_b\bm{v}^{T}\mathbf{S}\bm{u},
\end{array}
$$
where
$$
\begin{array}{ll}
\displaystyle
\mathbf{S}=\frac{K_\alpha}{h_a^\alpha}\left(\mathbf{C}_b^{(\beta)}\otimes\mathbf{I}_a\right)\left(\mathbf{I}_b\otimes \mathbf{D}_a^{(\alpha)}\right)
+\frac{K_\beta}{h_b^\beta}\left(\mathbf{I}_b\otimes \mathbf{C}_a^{(\alpha)}\right)\left(\mathbf{D}_b^{(\beta)}\otimes \mathbf{I}_a\right),
\end{array}
$$
$$
\begin{array}{ll}
\displaystyle
\mathbf{D}_p^{(\gamma)}=-\frac{1}{2\cos\left(\frac{\pi\gamma}{2}\right)}\left(\mathbf{E}_p^{(\gamma)}+\mathbf{E}_p^{(\gamma)^{T}}\right),\;
p=a,b,\;\gamma=\alpha,\beta.
\end{array}
$$
Here $\mathbf{I}_p$ is the identity matrix of order $M_p-1$,
$\mathbf{C}_p^{(\gamma)}$ $(p=a,b,\;\gamma=\alpha,\beta)$ has the form
$$\mathbf{C}_p^{(\gamma)}=
\left(
  \begin{array}{ccccc}
   1-2\sigma_{2}^{(\gamma)}& \sigma_{2}^{(\gamma)}&0& \cdots & 0 \vspace{0.2 cm}\\
   \sigma_{2}^{(\gamma)} & 1-2\sigma_{2}^{(\gamma)} & \sigma_{2}^{(\gamma)} & 0 & \cdots \vspace{0.2 cm}\\
    \vdots &\vdots & \ddots & \ddots & \ddots \vspace{0.2 cm}\\
  0 & \ldots & \sigma_{2}^{(\gamma)} &1-2\sigma_{2}^{(\gamma)}& \sigma_{2}^{(\gamma)}\vspace{0.2 cm}\\
 0& 0& \ldots & \sigma_{2}^{(\gamma)}& 1-2\sigma_{2}^{(\gamma)}\vspace{0.2 cm}\\
  \end{array}
\right)_{(M_p-1)\times(M_p-1).}
$$

It follows from Lemmas \ref{Le5.2} and \ref{Le5.3} that
$$
\begin{array}{lll}
\displaystyle
\mathbf{S}^{T}&=&\displaystyle\frac{K_\alpha}{h_a^\alpha}\left(\mathbf{I}_b\otimes \mathbf{D}_a^{(\alpha)}\right)^{T}\left(\mathbf{C}_b^{(\beta)}\otimes\mathbf{I}_a\right)^{T}
+\frac{K_\beta}{h_b^\beta}\left(\mathbf{D}_b^{(\beta)}\otimes \mathbf{I}_a\right)^{T}\left(\mathbf{I}_b\otimes \mathbf{C}_a^{(\alpha)}\right)^{T}\vspace{0.2 cm}\\
&=&\displaystyle\frac{K_\alpha}{h_a^\alpha}\left(\mathbf{I}_b\otimes \mathbf{D}_a^{(\alpha)}\right)\left(\mathbf{C}_b^{(\beta)}\otimes\mathbf{I}_a\right)
+\frac{K_\beta}{h_b^\beta}\left(\mathbf{D}_b^{(\beta)}\otimes \mathbf{I}_a\right)\left(\mathbf{I}_b\otimes \mathbf{C}_a^{(\alpha)}\right)\vspace{0.2 cm}\\
&=&\displaystyle\frac{K_\alpha}{h_a^\alpha}\left(\mathbf{C}_b^{(\beta)}\otimes\mathbf{I}_a\right)\left(\mathbf{I}_b\otimes \mathbf{D}_a^{(\alpha)}\right)
+\frac{K_\beta}{h_b^\beta}\left(\mathbf{I}_b\otimes \mathbf{C}_a^{(\alpha)}\right)\left(\mathbf{D}_b^{(\beta)}\otimes \mathbf{I}_a\right)
\vspace{0.2 cm}\\
&=&\displaystyle\mathbf{S},
\end{array}
$$
namely, $\mathbf{S}$ is a real symmetric matrix.

Besides, one easily knows that all the eigenvalues of matrix $\mathbf{C}_p^{(\gamma)}$ are
$$
\begin{array}{lll}
\displaystyle
\lambda_j\left(\mathbf{C}_p^{(\gamma)}\right)=1-4\sigma_{2}^{(\gamma)}\sin^2\left(\frac{\pi j}{2M_p}\right)\geq
\frac{4\sqrt{6}}{3}-3>0,\;j=0,1,\ldots,M_p-1.
\end{array}
$$
From Lemma \ref{Le5.4}, we see that the eigenvalues of matrix $\mathbf{D}_p^{(\gamma)}$ satisfy
$$
\begin{array}{lll}
\displaystyle
\lambda_j\left(\mathbf{D}_p^{(\gamma)}\right)\leq0,\;j=0,1,\ldots,M_p-1.
\end{array}
$$

Combining the above analysis and using Lemma \ref{Le5.1}, we can state that matrix $\mathbf{S}$ is
 semi-negative definite. Accordingly, there exists an orthogonal matrix $\mathbf{H}$ and a diagonal matrix $\mathbf{\Lambda}$ such that
$$
\begin{array}{lll}
\displaystyle
\mathbf{S}=-\mathbf{H}^{T}\mathbf{\Lambda}\mathbf{H}
=-\left(\mathbf{\Lambda}^{\frac{1}{2}}\mathbf{H}\right)^{T}\left(\mathbf{\Lambda}^{\frac{1}{2}}\mathbf{H}\right)
=-\mathbf{P}^{T}\mathbf{P},
\end{array}
$$
where $\mathbf{P}=\mathbf{\Lambda}^{\frac{1}{2}}\mathbf{H}$.
Therefore, we have
$$
\begin{array}{ll}
\displaystyle
((K_\alpha\mathcal{L}_y
\delta_x^\alpha+K_\beta\mathcal{L}_x\delta_y^\beta)\bm{u},\bm{v})=h_ah_b\bm{v}^{T}\mathbf{S}\bm{u}
=-h_ah_b\left(\mathbf{P}\bm{v}\right)^{T}\left(\mathbf{P}\bm{u}\right)=-(\mathcal{P}\bm{u},\mathcal{P}\bm{v}),
\end{array}
$$
where $\mathcal{P}$ is the associate operator of matrix $\mathbf{P}$, which is
 symmetric and semi-positive definite. The proof is thus ended.
\end{proof}

\begin{lemma}\label{Le5.6}
For any mesh functions, $\bm{u},\bm{v}\in \mathcal{V}_{h_a,h_b}$, there exists a symmetric positive
 definite operator $\mathcal{Q}$ such that
$$
\begin{array}{ll}
\displaystyle
(\mathcal{L}_x\mathcal{L}_y
\bm{u},\bm{v})=(\mathcal{Q}\bm{u},\mathcal{Q}\bm{v}).
\end{array}
$$
\end{lemma}

\begin{proof}
Similar to Lemma \ref{Le5.5}, the matrix form of inner product $(\mathcal{L}_x\mathcal{L}_y
\bm{u},\bm{v})$ is
\begin{equation}\label{eq39}
\displaystyle
(\mathcal{L}_x\mathcal{L}_y
\bm{u},\bm{v})=h_ah_b\bm{v}^{T}\mathbf{T}\bm{u},
\end{equation}
where $\mathbf{T}=\left(\mathbf{I}_b\otimes \mathbf{C}_a^{(\alpha)}\right)\left(\mathbf{C}_b^{(\beta)}\otimes\mathbf{I}_a\right)$.
Here, we easily know that matrix $\mathbf{T}$ is symmetric and positive definite by almost the same reasoning of Lemma \ref{Le5.5}.
So, there exist an orthogonal matrix $\mathbf{\widetilde{H}}$ and a diagonal matrix $\mathbf{\widetilde{\Lambda}}$ such that
\begin{equation}\label{eq40}
\displaystyle
\mathbf{T}=\mathbf{H}^{T}\mathbf{\Lambda}\mathbf{H}
=-\left(\mathbf{\Lambda}^{\frac{1}{2}}\mathbf{\widetilde{H}}\right)^{T}\left(\mathbf{\widetilde{\Lambda}}^{\frac{1}{2}}\mathbf{\widetilde{H}}\right)
=\mathbf{Q}^{T}\mathbf{Q},
\end{equation}
in which $\mathbf{Q}=\mathbf{\widetilde{\Lambda}}^{\frac{1}{2}}\mathbf{\widetilde{H}}$.

Substitution (\ref{eq40}) in (\ref{eq39}) yields
$$
\begin{array}{ll}
\displaystyle
(\mathcal{L}_x\mathcal{L}_y
\bm{u},\bm{v})=h_ah_b\bm{v}^{T}\mathbf{T}\bm{u}=h_ah_b\left(\mathbf{Q}\bm{v}\right)^{T}\mathbf{Q}\bm{u}
=(\mathcal{Q}\bm{u},\mathcal{Q}\bm{v}),
\end{array}
$$
where $\mathcal{Q}$ is a
symmetric and positive definite corresponding to matrix $\mathbf{Q}$.
The proof is completed.
\end{proof}

\begin{lemma}\label{Le5.7}
For any mesh function $\bm{u}\in \mathcal{V}_{h_a,h_b}$ and symmetric positive
 definite operator $\mathcal{Q}$, there holds that
$$
\begin{array}{ll}
\displaystyle
\left(\frac{4\sqrt{6}}{3}-3\right)||\bm{u}||\leq||\mathcal{Q}\bm{u}||\leq||\bm{u}||.
\end{array}
$$
\end{lemma}

\begin{proof}
From Lemma \ref{Le5.5}, we know that operators $\mathcal{L}_x$ and $\mathcal{L}_y$
are both symmetric and positive definite. So, there exist two symmetric and positive definite
operators $\mathcal{R}_x$ and $\mathcal{R}_y$ such that
$$
\begin{array}{ll}
\displaystyle
(\mathcal{L}_x\bm{u},\bm{v})=(\mathcal{R}_x\bm{u},\mathcal{R}_x\bm{v}),\;\;
(\mathcal{L}_y\bm{u},\bm{v})=(\mathcal{R}_y\bm{u},\mathcal{R}_y\bm{v})
\end{array}
$$
for any mesh functions $\bm{u},\bm{v}\in\mathcal{V}_{h_a,h_b}$.

Therefore, one has
$$
\begin{array}{lll}
\displaystyle
||\mathcal{Q}\bm{u}||^2&=&\displaystyle(\mathcal{L}_x\mathcal{L}_y\bm{u},\bm{u})
=(\mathcal{R}_x\mathcal{L}_y\bm{u},\mathcal{R}_x\bm{u})\vspace{0.2 cm}\\
&\geq&\displaystyle\left(\frac{4\sqrt{6}}{3}-3\right)(\mathcal{R}_x\bm{u},\mathcal{R}_x\bm{u})
=\left(\frac{4\sqrt{6}}{3}-3\right)(\mathcal{L}_x\bm{u},\bm{u})\vspace{0.2 cm}\\
&\geq&\displaystyle\left(\frac{4\sqrt{6}}{3}-3\right)^2||\bm{u}||^2.
\end{array}
$$
On the other hand, we also have
$$
\begin{array}{lll}
\displaystyle
||\mathcal{Q}\bm{u}||^2&=&\displaystyle(\mathcal{L}_x\mathcal{L}_y\bm{u},\bm{u})
=(\mathcal{R}_x\mathcal{L}_y\bm{u},\mathcal{R}_x\bm{u})\vspace{0.2 cm}\\
&\leq&\displaystyle(\mathcal{R}_x\bm{u},\mathcal{R}_x\bm{u})
=(\mathcal{L}_x\bm{u},\bm{u})\vspace{0.2 cm}\\
&\leq&\displaystyle||\bm{u}||^2.
\end{array}
$$
The proof is thus shown.
\end{proof}

Now, we turn to consider the stability of difference scheme (\ref{eq36})--(\ref{eq38}). Suppose $v_{i,j}^k$ is the solution
of the following finite difference equation,
\begin{eqnarray}\label{eq41}
 &&\mathcal{L}_x\mathcal{L}_y\delta_t v_{i,j}^{k+\frac{1}{2}}
 =-\mathcal{L}_x\mathcal{L}_y v_{i,j}^{k+\frac{1}{2}}+K_\alpha\mathcal{L}_y\delta_x^\alpha
v_{i,j}^{k+\frac{1}{2}} +K_\beta\mathcal{L}_x\delta_y^\beta
v_{i,j}^{k+\frac{1}{2}}
 \vspace{0.2 cm}\\&& +\mathcal{L}_x\mathcal{L}_yf_{i,j}^{k+\frac{1}{2}},1\leq i\leq M_x-1,\;1\leq j\leq M_y-1,\;0\leq k\leq N-1,\nonumber
  \end{eqnarray}
  \begin{eqnarray}\label{eq42}
v_{i,j}^{0}=u^0(x_i,y_j)+\varepsilon_{i,j},\;\;(x_i,y_j)\in{\overline{\Omega}}_h,
  \end{eqnarray}
  \begin{eqnarray}\label{eq43}
v_{i,j}^{k}=0,\;(x_i,y_j)\in\partial{\Omega}_h,\;1\leq k\leq N.
\end{eqnarray}
Let $\xi_{i,j}^k=v_{i,j}^{k}-u_{i,j}^{k}$, then the perturbation equation can
 be obtained by using equations (\ref{eq36})--(\ref{eq38}) and (\ref{eq41})--(\ref{eq43}),
\begin{eqnarray}\label{eq44}
 \mathcal{L}_x\mathcal{L}_y\delta_t \xi_{i,j}^{k+\frac{1}{2}}
 =-\mathcal{L}_x\mathcal{L}_y \xi_{i,j}^{k+\frac{1}{2}}+K_\alpha\mathcal{L}_y\delta_x^\alpha
\xi_{i,j}^{k+\frac{1}{2}} +K_\beta\mathcal{L}_x\delta_y^\beta
\xi_{i,j}^{k+\frac{1}{2}},\vspace{0.2 cm}\\1\leq i\leq M_x-1,\;1\leq j\leq M_y-1,\;0\leq k\leq N-1,\nonumber
\end{eqnarray}
$$
\begin{array}{lll} \displaystyle
\xi_{i,j}^{0}=\varepsilon_{i,j},\;\;(x_i,y_j)\in{\overline{\Omega}}_h,\vspace{0.2 cm}\\ \displaystyle
\xi_{i,j}^{k}=0,\;(x_i,y_j)\in\partial{\Omega}_h,\;1\leq k\leq N.
\end{array}
$$

Next, the stability is shown below.

\begin{theorem}
The finite difference scheme (\ref{eq36})--(\ref{eq38}) is unconditionally stable with respect to the initial values.
\end{theorem}

\begin{proof}
Taking the inner product of (\ref{eq44}) with $\bm{\xi}^{k+\frac{1}{2}}$ gives
\begin{eqnarray}\label{eq45}
&&\left(\mathcal{L}_x\mathcal{L}_y\delta_t \bm{\xi}^{k+\frac{1}{2}},\bm{\xi}^{k+\frac{1}{2}}\right)
 =-\left(\mathcal{L}_x\mathcal{L}_y\bm{\xi}^{k+\frac{1}{2}},\bm{\xi}^{k+\frac{1}{2}}\right)\\&&
 +\left(\left(K_\alpha\mathcal{L}_y\delta_x^\alpha+K_\beta\mathcal{L}_x\delta_y^\beta\right)\bm{\xi}^{k+\frac{1}{2}},
 \bm{\xi}^{k+\frac{1}{2}}\right).\nonumber
\end{eqnarray}

Applying Lemma \ref{Le5.6} to the first term on the right hand side of equation (\ref{eq45}), we know that there exists
 a symmetric and positive definite operator $\mathcal{Q}$ such that,
\begin{eqnarray}\label{eq46}
 -\left(\mathcal{L}_x\mathcal{L}_y\bm{\xi}^{k+\frac{1}{2}},\bm{\xi}^{k+\frac{1}{2}}\right)=
 -\left(\mathcal{Q}\bm{\xi}^{k+\frac{1}{2}},\mathcal{Q}\bm{\xi}^{k+\frac{1}{2}}\right)=
 -\left\|\mathcal{Q}\bm{\xi}^{k+\frac{1}{2}}\right\|^2\leq0.
\end{eqnarray}

For the second term on the right hand side of (\ref{eq45}),
 we can also get the following result by using Lemma \ref{Le5.5},
\begin{eqnarray}\label{eq47}
\left(\left(K_\alpha\mathcal{L}_y\delta_x^\alpha+K_\beta\mathcal{L}_x\delta_y^\beta\right)\bm{\xi}^{k+\frac{1}{2}},
 \bm{\xi}^{k+\frac{1}{2}}\right)=-\left\|\mathcal{P}\bm{\xi}^{k+\frac{1}{2}}\right\|^2\leq0,
\end{eqnarray}
where $\mathcal{P}$ is a symmetric and positive definite operator.

Substituting (\ref{eq46}) and (\ref{eq47}) into (\ref{eq45}) leads to
\begin{eqnarray}\label{eq48}
\left\|\mathcal{Q}\bm{\xi}^{k+1}\right\|^2-\left\|\mathcal{Q}\bm{\xi}^{k}\right\|^2\leq0.
\end{eqnarray}
Replacing $k$ by $n$ and summing up $n$ from 0 to $k-1$ on both sides of (\ref{eq48}) yield
$$
\begin{array}{lll} \displaystyle
\left\|\mathcal{Q}\bm{\xi}^{k}\right\|\leq\left\|\mathcal{Q}\bm{\xi}^{0}\right\|.
\end{array}
$$
Utilizing Lemma \ref{Le5.7} gives
$$
\begin{array}{lll} \displaystyle
\left\|\bm{\xi}^{k}\right\|\leq\frac{(4\sqrt{6}+9)}{5}\left\|\bm{\xi}^{0}\right\|
=\frac{(4\sqrt{6}+9)}{5}\left\|\bm{\varepsilon}\right\|.
\end{array}
$$
All this ends the proof.
\end{proof}

Finally, we study the convergence for finite difference scheme (\ref{eq36})--(\ref{eq38}).

\begin{theorem}
 Let $u_{i,j}^k$ and $u(x,y, t)$ be the solutions of the finite difference
 scheme (\ref{eq36})--(\ref{eq38}) and problem (\ref{eq28})--(\ref{eq30}), respectively.
 Denote $e_{i,j}^k=u_{i,j}^k-u(x_i,y_j,t_k)$, $\;0\leq k\leq N,\;\;0\leq i\leq M_a\;\;0\leq j\leq M_b$.
Then there holds
$$
\begin{array}{lll} \displaystyle
||\bm{e}^{k}||\leq 32C_2\sqrt{TL_aL_b}\exp\left(
8T\right)(\tau^2+h_a^3+h_b^3),
\end{array}
$$
where $C_2$ is a positive constant.
\end{theorem}

\begin{proof}
Subtracting (\ref{eq36})--(\ref{eq38}) from (\ref{eq28})--(\ref{eq30}) leads to an
error system as follows,
\begin{eqnarray}\label{eq49}
 \mathcal{L}_x\mathcal{L}_y\delta_t e_{i,j}^{k+\frac{1}{2}}
 =-\mathcal{L}_x\mathcal{L}_ye_{i,j}^{k+\frac{1}{2}}+K_\alpha\mathcal{L}_y\delta_x^\alpha
e_{i,j}^{k+\frac{1}{2}} +K_\beta\mathcal{L}_x\delta_y^\beta
e_{i,j}^{k+\frac{1}{2}}
  +R_{i,j}^{k},\vspace{0.2 cm}\\1\leq i\leq M_a-1,\;1\leq j\leq M_b-1,\;0\leq k\leq N-1,\nonumber
  \end{eqnarray}
$$
\begin{array}{lll} \displaystyle
u_{i,j}^{0}=0,\;\;(x_i,y_j)\in{\overline{\Omega}}_h,\vspace{0.2 cm}\\ \displaystyle
u_{i,j}^{k}=0,\;(x_i,y_j)\in\partial{\Omega}_h,\;1\leq k\leq N.
\end{array}
$$

Taking the inner product of (\ref{eq49}) with $\bm{e}^{k+\frac{1}{2}}$ and
using the similar method in the proof of Theorem 4.2, we reach that
$$
\begin{array}{lll} \displaystyle
\left\|\mathcal{Q}\bm{e}^{k+1}\right\|^2\leq \displaystyle\left\|\mathcal{Q}\bm{e}^{k}\right\|^2
+2\tau\left(\bm{R}^{k+1},\bm{e}^{k+\frac{1}{2}}\right)\vspace{0.2 cm}\\
\leq \displaystyle\left\|\mathcal{Q}\bm{e}^{k}\right\|^2+8\tau\left\|\bm{R}^{k+1}\right\|^2
+\frac{1}{16}\tau\left(\left\|\bm{e}^{k}\right\|^2+\left\|\bm{e}^{k+1}\right\|^2\right)
\vspace{0.2 cm}\\
\leq \displaystyle\left\|\mathcal{Q}\bm{e}^{k}\right\|^2
+\frac{1}{16}\tau\left(\left\|\bm{e}^{k}\right\|^2+\left\|\bm{e}^{k+1}\right\|^2\right)+8\tau C_2^2L_aL_b\left(\tau^2+h_a^3+h_b^3\right)^2.
\end{array}
$$
Replacing $k$ by $n$ and summing up $n$ from 0 to $k-1$ on both sides of the above inequality
 and following Lemma \ref{Le5.7}, we reach that
$$
\begin{array}{lll} \displaystyle
||\bm{e}^{k}||^2&\leq&\displaystyle
\frac{18}{(16\sqrt{6}-33)(16\sqrt{6}-39)}\tau\sum_{n=0}^{k-1}||\bm{e}^{n}||^2\vspace{0.2 cm}\\
&&\displaystyle+\frac{1152}{(16\sqrt{6}-33)(16\sqrt{6}-39)}C_2^2TL_aL_b(\tau^2+h_a^3+h_b^3)^2\vspace{0.2 cm}\\
&\leq&\displaystyle
16\tau\sum_{n=0}^{k-1}||\bm{e}^{n}||^2
+1024C_2^2TL_aL_b(\tau^2+h_a^3+h_b^3)^2.
\end{array}
$$

So we have the following result via Lemma \ref{Le4.3},
$$
\begin{array}{lll} \displaystyle
||\bm{e}^{k}||\leq
32c_2\sqrt{TL_aL_b}\exp\left(
8T\right)(\tau^2+h_a^3+h_b^3).
\end{array}
$$
The proof is finally shown in the end.
\end{proof}

\section{Numerical examples}

In this section, the validity and convergence orders of the numerical
 algorithms constructed in this paper are demonstrated by several numerical tests.
\begin{example}\label{Ex1}
Consider function $u(x)=x^2(1-x)^2.$ The exact expression at $x=0.5$ is given by
$$
\begin{array}{lll}
\displaystyle \frac{\partial^\alpha
u(x)}{\partial{|x|^\alpha}}\mid_{x=0.5}&= &\displaystyle -\frac{1}{2\cos\left(\frac{\pi\alpha}{2}\right)}
\left\{\frac{\Gamma(3)}{\Gamma(3-\alpha)}\left(\frac{1}{2}\right)^{1-\alpha}
-\frac{2\Gamma(4)}{\Gamma(4-\alpha)}\left(\frac{1}{2}\right)^{2-\alpha}\right.\vspace{0.2 cm}\\&&\displaystyle
\left.+\frac{\Gamma(5)}{\Gamma(5-\alpha)}\left(\frac{1}{2}\right)^{3-\alpha}\right\},\;\;\;x\in[0,1].
\end{array}
$$
\end{example}

Choosing different spatial stepsizes $h$, we compute Riesz derivative of function $u(x)$ using
 numerical formulas (\ref{eq7}), (\ref{eq8}) and (\ref{eq9}), respectively. Tables \ref{Tab1}, \ref{Tab2},
  and \ref{Tab3} list the absolute errors and
 numerical convergence orders at $x=0.5$ for different orders $\alpha$ in (1,2). From these results,
 one can see that the numerical results are in line with the theoretical order.

\begin{table}\renewcommand\arraystretch{1.1}
 \begin{center}
 \caption{ The absolute errors and convergence orders of Example \ref{Ex1}
  by numerical formula (\ref{eq7}).}\label{Tab1}
 \vspace{0.2 cm}
 \begin{footnotesize}
\begin{tabular}{c c c c c c }\hline
  $\alpha$ &\; $h$\;&  \textrm{the absolute errors}&\;\;\;\;$\textrm {the convergence orders}$
  \\
  \hline \vspace{0.1 cm}
  $1.1 $& $\frac{1}{20}$ & 1.740717e-04	 \;\; &  ---\\ \vspace{0.1 cm}
  $$  & $\frac{1}{40}$&   2.185595e-05	 \;\; &	2.9936 \\
  $$& $\frac{1}{80}$ &  2.742123e-06	 \;\; &	2.9947 \\
  $$& $\frac{1}{160}$ &  3.434158e-07	 \;\; &	2.9973\\
   $$& $\frac{1}{320}$ &  4.296784e-08	 \;\; &	2.9986\\\hline
   $1.3 $& $\frac{1}{20}$ & 1.756079e-04 \;\;& ---\\ \vspace{0.1 cm}
  $$  & $\frac{1}{40}$& 2.198613e-05\;\;&		2.9977\\
   $$  & $\frac{1}{80}$&2.751417e-06	\;\;&	2.9983\\
  \vspace{0.1 cm}
  $$& $\frac{1}{160}$ &3.441531e-07	\;\;&	2.9991\\
   $$& $\frac{1}{320}$ & 4.303416e-08\;\;&		2.9995\\\hline
  $1.5 $& $\frac{1}{20}$ & 1.377134e-04	\;\; &     ---\\ \vspace{0.1 cm}
  $$  & $\frac{1}{40}$&  1.716087e-05	\;\; &  		3.0045\\
  $$  & $\frac{1}{80}$&2.143372e-06		\;\; &  	3.0012\\
  $$& $\frac{1}{160}$ & 2.678606e-07	\;\; &  		3.0003  \\
   $$& $\frac{1}{320}$ & 3.348027e-08		\;\; &  	3.0001\\\hline
   $1.7 $& $\frac{1}{20}$ &7.211650e-05	 \;\; &      ---\\ \vspace{0.1 cm}
  $$  & $\frac{1}{40}$& 8.991024e-06 \;\; & 		3.0038 \\
   $$  & $\frac{1}{80}$& 1.123719e-06	 \;\; & 	3.0002 \\
  $$& $\frac{1}{160}$ & 1.404937e-07	 \;\; & 	2.9997\\
   $$& $\frac{1}{320}$ &  1.756457e-08	 \;\; & 	2.9998\\ \hline
$1.9 $& $\frac{1}{20}$ &	1.056422e-05	 \;\; &      ---\\ \vspace{0.1 cm}
  $$  & $\frac{1}{40}$& 1.364672e-06	 \;\; &	2.9526 \\
   $$  & $\frac{1}{80}$& 1.735867e-07	 \;\; &	2.9748\\
  $$& $\frac{1}{160}$ & 2.189369e-08	 \;\; &	2.9871\\
   $$& $\frac{1}{320}$ & 2.749249e-09	 \;\; &	2.9934\\\hline
\end{tabular}
 \end{footnotesize}
 \end{center}
 \end{table}

\begin{table}\renewcommand\arraystretch{1.1}
 \begin{center}
 \caption{  The absolute errors and convergence orders of Example \ref{Ex1}
  by numerical formula (\ref{eq8}).}\label{Tab2}
 \vspace{0.2 cm}
 \begin{footnotesize}
\begin{tabular}{c c c c c c }\hline
  $\alpha$ &\; $h$\;&  \textrm{the absolute errors}&\;\;\;\;$\textrm {the convergence orders}$
  \\
  \hline \vspace{0.1 cm}
  $1.1 $& $\frac{1}{20}$ & 5.290778e-02		 \;\; &  ---\\ \vspace{0.1 cm}
  $$  & $\frac{1}{40}$&  6.548041e-03 \;\; & 		3.0143 \\
  $$& $\frac{1}{80}$ & 8.150577e-04	 \;\; & 	3.0061 \\
  $$& $\frac{1}{160}$ &  1.016718e-04 \;\; & 		3.0030\\
   $$& $\frac{1}{320}$ & 1.269602e-05 \;\; & 		3.0015\\ \hline
   $1.3 $& $\frac{1}{20}$ & 2.033985e-02	 \;\;& ---\\ \vspace{0.1 cm}
  $$  & $\frac{1}{40}$& 2.512801e-03 \;\;& 		3.0169\\
   $$  & $\frac{1}{80}$&3.117365e-04 \;\;& 		3.0109\\
  $$& $\frac{1}{160}$ &3.881109e-05	 \;\;& 	3.0058\\
   $$& $\frac{1}{320}$ & 4.841522e-06 \;\;& 		3.0029\\\hline
  $1.5 $& $\frac{1}{20}$ & 1.263828e-02		\;\; &     ---\\ \vspace{0.1 cm}
  $$  & $\frac{1}{40}$&  1.583605e-03	\;\; &	2.9965\\
  $$  & $\frac{1}{80}$&1.966563e-04	\;\; &	3.0095\\
  $$& $\frac{1}{160}$ & 2.448135e-05	\;\; &	3.0059 \\
   $$& $\frac{1}{320}$ &3.053477e-06	\;\; &	3.0032\\ \hline
   $1.7 $& $\frac{1}{20}$ &7.701877e-03		 \;\; &      ---\\ \vspace{0.1 cm}
  $$  & $\frac{1}{40}$&9.893186e-04	 \;\; &	2.9607 \\
   $$  & $\frac{1}{80}$&1.233352e-04 \;\; &		3.0039 \\
  $$& $\frac{1}{160}$ &1.537077e-05	 \;\; &	3.0043\\
   $$& $\frac{1}{320}$ &  1.917897e-06	 \;\; &	3.0026\\\hline
$1.9 $& $\frac{1}{20}$ &2.787724e-03		 \;\; &      ---\\ \vspace{0.1 cm}
  $$  & $\frac{1}{40}$& 3.697284e-04 \;\; & 		2.9145 \\
   $$  & $\frac{1}{80}$&4.637689e-05	 \;\; & 	2.9950\\
  $$& $\frac{1}{160}$ & 5.791288e-06	 \;\; & 	3.0014\\
   $$& $\frac{1}{320}$ &7.231609e-07	 \;\; & 	3.0015\\
 \hline
\end{tabular}
 \end{footnotesize}
 \end{center}
 \end{table}

\begin{table}\renewcommand\arraystretch{1.1}
 \begin{center}
 \caption{ The absolute errors and convergence orders of Example \ref{Ex1}
  by numerical formula (\ref{eq9}).}\label{Tab3}
 \vspace{0.2 cm}
 \begin{footnotesize}
\begin{tabular}{c c c c c c }\hline
  $\alpha$ &\; $h$\;&  \textrm{the absolute errors}&\;\;\;\;$\textrm {the convergence orders}$
  \\
  \hline \vspace{0.1 cm}
  $1.1 $& $\frac{1}{20}$ & 8.281680e-07	\;\; &  ---\\ \vspace{0.1 cm}
  $$  & $\frac{1}{40}$&  5.167207e-08\;\; &		4.0025 \\
  $$& $\frac{1}{80}$ & 3.218255e-09	\;\; &	4.0050 \\
  $$& $\frac{1}{160}$ &  2.007975e-10	\;\; &	4.0025\\\hline
   $1.3 $& $\frac{1}{20}$ & 8.898742e-07		 \;\;& ---\\ \vspace{0.1 cm}
  $$  & $\frac{1}{40}$& 5.777396e-08	 \;\;&	3.9451\\
   $$  & $\frac{1}{80}$&3.654194e-09 \;\;&		3.9828\\
  $$& $\frac{1}{160}$ &2.294926e-10	 \;\;&	3.9930\\\hline
  $1.5 $& $\frac{1}{20}$ & 5.084772e-07			\;\; &     ---\\ \vspace{0.1 cm}
  $$  & $\frac{1}{40}$&  3.725522e-08		\;\; &  	3.7707\\
  $$  & $\frac{1}{80}$&2.454356e-09		\;\; &  	3.9240\\
  $$& $\frac{1}{160}$ & 1.567338e-10	\;\; &  		3.9690 \\\hline
   $1.7 $& $\frac{1}{20}$ &1.822972e-07			 \;\; &      ---\\ \vspace{0.1 cm}
  $$  & $\frac{1}{40}$&1.692478e-08	 \;\; &	3.4291 \\
   $$  & $\frac{1}{80}$&1.191028e-09 \;\; &		3.8289 \\
  $$& $\frac{1}{160}$ &7.878076e-11	 \;\; &	3.9182\\\hline
$1.9 $& $\frac{1}{20}$ &9.867011e-08  \;\; &   ---\\ \vspace{0.1 cm}
  $$  & $\frac{1}{40}$& 7.533874e-09	 \;\; &	3.7111 \\
   $$  & $\frac{1}{80}$&5.041596e-10	 \;\; &	3.9014\\
  $$& $\frac{1}{160}$ & 3.322587e-11	 \;\; &	3.9235\\\hline
\end{tabular}
 \end{footnotesize}
 \end{center}
 \end{table}

\begin{example}\label{Ex2}
We consider the following
 one-dimensional Riesz spatial fractional reaction-dispersion equation,
  $$\left\{
\begin{array}{lll}
\displaystyle
 \frac{\partial
u(x,t)}{\partial t}=-u(x,t)+e^{-12}
 \frac{\partial^\alpha
u(x,t)}{\partial{|x|^\alpha}}+f(x,t),\;(x,t)\in(0,1)\times(0,1],\vspace{0.2 cm}\\ \displaystyle
u(x,0)=0,\;\;x\in(0,1),\vspace{0.2 cm}\\ \displaystyle
u(0,t)=u(1,t)=0,\;t\in(0,1],
\end{array}\right.
$$
where the source term $f(x,t)$ is
$$
\begin{array}{lll}
\displaystyle
2e^{t} x^6(1-x)^6+
\frac{\sin t}{2\cos\left(\frac{\pi}{2}\alpha\right)}
\left\{\frac{\Gamma(7)}{\Gamma(7-\alpha)}\left[x^{6-\alpha}+(1-x)^{6-\alpha}\right]\right. \vspace{0.2 cm}\\ \displaystyle
-\frac{6\Gamma(8)}{\Gamma(8-\alpha)}\left[x^{7-\alpha}+(1-x)^{7-\alpha}\right]
+\frac{15\Gamma(9)}{\Gamma(9-\alpha)}\left[x^{8-\alpha}+(1-x)^{8-\alpha}\right]
\vspace{0.2 cm}\\ \displaystyle-
\frac{20\Gamma(10)}{\Gamma(10-\alpha)}\left[x^{9-\alpha}+(1-x)^{9-\alpha}\right]
+\frac{15\Gamma(11)}{\Gamma(11-\alpha)}\left[x^{10-\alpha}+(1-x)^{10-\alpha}\right]
\vspace{0.2 cm}\\ \displaystyle\left.
-\frac{6\Gamma(12)}{\Gamma(12-\alpha)}\left[x^{11-\alpha}+(1-x)^{11-\alpha}\right]
+\frac{\Gamma(13)}{\Gamma(13-\alpha)}\left[x^{12-\alpha}+(1-x)^{12-\alpha}\right]\right\}.
\end{array}
$$
The exact solution of this equation is $u(x,t)=e^{t} x^6(1-x)^6$ and satisfies the according
initial and boundary values conditions.
\end{example}

\begin{table}\renewcommand\arraystretch{1.1}
 \begin{center}
 \caption{ The absolute errors (TAEs), temporal convergence
order (TCO) and spatial convergence order (SCO) of Example \ref{Ex2} by difference scheme (\ref{eq15})--(\ref{eq17}).}\label{Tab.4}
 \vspace{0.2 cm}
 \begin{footnotesize}
\begin{tabular}{c c c c c c }\hline
   & \;\;\;\; \;\; &\;\;\;\;\;\;&\;\;\;\;\;\;&\;\;\;\;\;\;&\;\;\;\;\;\;\\
  $\alpha$ &\;$\tau$,\;$h$\;&  \;\;TAEs\;\;\;\; &\;\;\;\; TCO\;\;\;\; & SCO\\\hline \vspace{0.1 cm}
  $1.1 $& $\tau=\frac{1}{4},h=\frac{1}{4}$ &     2.984674e-06    &  ---&  ---\\ \vspace{0.1 cm}
  $$  & $\tau=\frac{\sqrt{2}}{16},h=\frac{1}{8}$&    3.613655e-07	&	2.0307	&	3.0460\\ \vspace{0.1 cm}
  $$& $\tau=\frac{1}{32},h=\frac{1}{16}$ &      4.685713e-08	&	1.9647	&	2.9471\\ \vspace{0.1 cm}
$$  & $\tau=\frac{\sqrt{2}}{128},h=\frac{1}{32}$&      5.813993e-09	&	2.0071	&	3.0107\\ \vspace{0.1 cm}
  $$& $\tau=\frac{1}{256},h=\frac{1}{64}$ &       7.321694e-10	&	1.9929	&	2.9893\\ \hline\vspace{0.1 cm}
 $1.3 $& $\tau=\frac{1}{4},h=\frac{1}{4}$ &    2.984597e-06		&  ---&  ---\\ \vspace{0.1 cm}
  $$  & $\tau=\frac{\sqrt{2}}{16},h=\frac{1}{8}$&    3.617522e-07	&	2.0296	&	3.0445\\ \vspace{0.1 cm}
  $$& $\tau=\frac{1}{32},h=\frac{1}{16}$ &      4.690406e-08	&	1.9648		&2.9472\\ \vspace{0.1 cm}
$$  & $\tau=\frac{\sqrt{2}}{128},h=\frac{1}{32}$&     5.819491e-09	&	2.0072	&	3.0107\\ \vspace{0.1 cm}
  $$& $\tau=\frac{1}{256},h=\frac{1}{64}$ &       7.328387e-10	&	1.9929	&	2.9893\\ \hline\vspace{0.1 cm}
 $1.5 $& $\tau=\frac{1}{4},h=\frac{1}{4}$ &     2.981516e-06		&  ---&  ---\\ \vspace{0.1 cm}
  $$  & $\tau=\frac{\sqrt{2}}{16},h=\frac{1}{8}$&    3.616854e-07	&	2.0288	&	3.0432\\ \vspace{0.1 cm}
  $$& $\tau=\frac{1}{32},h=\frac{1}{16}$ &      4.690789e-08	&	1.9646	&	2.9468\\ \vspace{0.1 cm}
$$  & $\tau=\frac{\sqrt{2}}{128},h=\frac{1}{32}$&    5.819848e-09	&	2.0072	&	3.0108\\ \vspace{0.1 cm}
  $$& $\tau=\frac{1}{256},h=\frac{1}{64}$ &       7.328573e-10	&	1.9929	&	2.9894\\ \hline\vspace{0.1 cm}
  $1.7 $& $\tau=\frac{1}{4},h=\frac{1}{4}$ &    2.974813e-06		&  ---&  ---\\ \vspace{0.1 cm}
  $$  & $\tau=\frac{\sqrt{2}}{16},h=\frac{1}{8}$&    3.609314e-07	&	2.0287	&	3.0430\\ \vspace{0.1 cm}
  $$& $\tau=\frac{1}{32},h=\frac{1}{16}$ &     4.683820e-08	&	1.9640		&2.9460\\ \vspace{0.1 cm}
$$  & $\tau=\frac{\sqrt{2}}{128},h=\frac{1}{32}$&    5.811874e-09	&	2.0071	&	3.0106\\ \vspace{0.1 cm}
  $$& $\tau=\frac{1}{256},h=\frac{1}{64}$ &       7.318735e-10	&	1.9929	&	2.9893\\ \hline\vspace{0.1 cm}
 $1.9 $& $\tau=\frac{1}{4},h=\frac{1}{4}$ &    2.963689e-06	    &  ---&  ---\\ \vspace{0.1 cm}
  $$  & $\tau=\frac{\sqrt{2}}{16},h=\frac{1}{8}$&   3.593385e-07	&	2.0293	&	3.0440\\ \vspace{0.1 cm}
  $$& $\tau=\frac{1}{32},h=\frac{1}{16}$ &     4.668265e-08	&	1.9629	&	2.9444\\ \vspace{0.1 cm}
$$  & $\tau=\frac{\sqrt{2}}{128},h=\frac{1}{32}$&    5.795636e-09	&	2.0066	&	3.0098\\ \vspace{0.1 cm}
  $$& $\tau=\frac{1}{256},h=\frac{1}{64}$ &     7.300171e-10	&	1.9926	&	2.9890\\ \hline\vspace{0.1 cm}
\end{tabular}
 \end{footnotesize}
 \end{center}
 \end{table}

Using numerical scheme (\ref{eq15})--(\ref{eq17}), we present the absolute errors and the corresponding space and time
convergence orders with different stepsizes in Table 4. It can be found that the convergence orders of
 scheme (\ref{eq15})--(\ref{eq17}) are almost second- and
 third-order in time and space directions, respectively, which is in agreement with the theoretical convergence order.

\begin{example}\label{Ex3}
We consider the following
 two-dimensional Riesz spatial fractional reaction-dispersion equation,
  $$\left\{
\begin{array}{lll}
\displaystyle
 \frac{\partial{{}u(x,y,t)}}{\partial{t}}
 =-u(x,y,t)+\pi^{-8}\left(\frac{\partial^\alpha
u(x,y,t)}{\partial{|x|^\alpha}}+\frac{\partial^\beta
u(x,y,t)}{\partial{|y|^\beta}}\right)
  +f(x,y,t),\\ \displaystyle\hspace{8.5 cm}(x,y;t)\in\Omega\times(0,1],\\ \displaystyle
u(x,y,0)=0,\;\;(x,y)\in\bar{\Omega},\vspace{0.2 cm}\\ \displaystyle
u(x,y,t)=0,\;(x,y;t)\in\partial\Omega\times(0,1],
\end{array}\right.
$$
where $\Omega=[0,1]\times[0,1]$, and the source term $f(x,y,t)$ is
$$
\begin{array}{lll}
\displaystyle
3e^{2t}x^6(1-x)^6y^6(1-y)^6+
\frac{e^{2t}y^6(1-y)^6}{2\cos\left(\frac{\pi}{2}\alpha\right)}
\left\{\frac{\Gamma(7)}{\Gamma(7-\alpha)}\left[x^{6-\alpha}+(1-x)^{6-\alpha}\right]\right. \vspace{0.2 cm}\\ \displaystyle
-\frac{6\Gamma(8)}{\Gamma(8-\alpha)}\left[x^{7-\alpha}+(1-x)^{7-\alpha}\right]
+\frac{15\Gamma(9)}{\Gamma(9-\alpha)}\left[x^{8-\alpha}+(1-x)^{8-\alpha}\right]
\vspace{0.2 cm}\\ \displaystyle-
\frac{20\Gamma(10)}{\Gamma(10-\alpha)}\left[x^{9-\alpha}+(1-x)^{9-\alpha}\right]
+\frac{15\Gamma(11)}{\Gamma(11-\alpha)}\left[x^{10-\alpha}+(1-x)^{10-\alpha}\right]
\vspace{0.2 cm}\\ \displaystyle\left.
-\frac{6\Gamma(12)}{\Gamma(12-\alpha)}\left[x^{11-\alpha}+(1-x)^{11-\alpha}\right]
+\frac{\Gamma(13)}{\Gamma(13-\alpha)}\left[x^{12-\alpha}+(1-x)^{12-\alpha}\right]\right\}\vspace{0.2 cm}\\ \displaystyle+
\frac{e^{2t}x^6(1-x)^6}{2\cos\left(\frac{\pi}{2}\beta\right)}
\left\{\frac{\Gamma(7)}{\Gamma(7-\beta)}\left[y^{6-\beta}+(1-y)^{6-\beta}\right]\right. \vspace{0.2 cm}\\ \displaystyle
-\frac{6\Gamma(8)}{\Gamma(8-\beta)}\left[y^{7-\beta}+(1-y)^{7-\beta}\right]
+\frac{15\Gamma(9)}{\Gamma(9-\beta)}\left[y^{8-\beta}+(1-y)^{8-\beta}\right]
\vspace{0.2 cm}\\ \displaystyle
-\frac{20\Gamma(10)}{\Gamma(10-\beta)}\left[y^{9-\beta}+(1-y)^{9-\beta}\right]
+\frac{15\Gamma(11)}{\Gamma(11-\beta)}\left[y^{10-\beta}+(1-y)^{10-\beta}\right]
\vspace{0.2 cm}\\ \displaystyle\left.
-\frac{6\Gamma(12)}{\Gamma(12-\beta)}\left[y^{11-\beta}+(1-y)^{11-\beta}\right]
+\frac{\Gamma(13)}{\Gamma(13-\beta)}\left[y^{12-\beta}+(1-y)^{12-\beta}\right]\right\}.
\end{array}
$$
The exact solution of this equation is $u(x,t)=e^{2t}x^6(1-x)^6y^6(1-y)^6$ and satisfies the corresponding
initial and boundary values conditions.
\end{example}

\begin{table}\renewcommand\arraystretch{1.1}\label{Tab5}
 \begin{center}
 \caption{ The absolute errors (TAEs), temporal convergence
order (TCO) and spatial convergence order (SCO) of Example \ref{Ex3} by difference scheme (\ref{eq36})--(\ref{eq38}).}\label{Tab.5}
 \vspace{0.2 cm}
 \begin{footnotesize}
\begin{tabular}{c c c c c c }\hline
   & \;\;\;\; \;\; &\;\;\;\;\;\;&\;\;\;\;\;\;&\;\;\;\;\;\;&\;\;\;\;\;\;\\
  $\alpha,\beta$ &\;$\tau$,\;$h_a,h_b$\;&  \;\;TAEs\;\;\;\; &\;\;\;\; TCO\;\;\;\; & SCO\\\hline \vspace{0.1 cm}
  $\alpha=1.1,\beta=1.8 $& $\tau=\frac{1}{4},h_a=h_b=\frac{1}{4}$ &    7.150284e-09	    &  ---&  ---\\ \vspace{0.1 cm}
  $$  & $\tau=\frac{\sqrt{2}}{16},h_a=h_b=\frac{1}{8}$&    8.680618e-10	&	2.0281	&	3.0421\\ \vspace{0.1 cm}
  $$& $\tau=\frac{1}{32},h_a=h_b=\frac{1}{16}$ &     1.155609e-10	&	1.9394	&	2.9091\\ \vspace{0.1 cm}
$$  & $\tau=\frac{\sqrt{2}}{128},h_a=h_b=\frac{1}{32}$&     1.428060e-11	&	2.0110	&	3.0165\\ \hline\vspace{0.1 cm}
 $\alpha=1.3,\beta=1.6$& $\tau=\frac{1}{4},h_a=h_b=\frac{1}{4}$ &    7.221370e-09		&  ---&  ---\\ \vspace{0.1 cm}
  $$  & $\tau=\frac{\sqrt{2}}{16},h_a=h_b=\frac{1}{8}$&    8.805609e-10	&	2.0239	&	3.0358\\ \vspace{0.1 cm}
  $$& $\tau=\frac{1}{32},h_a=h_b=\frac{1}{16}$ &      1.168858e-10	&	1.9422	&	2.9133\\ \vspace{0.1 cm}
$$  & $\tau=\frac{\sqrt{2}}{128},h_a=h_b=\frac{1}{32}$&     1.442860e-11	&	2.0121	&	3.0181\\ \hline\vspace{0.1 cm}
 $\alpha=1.5,\beta=1.5$& $\tau=\frac{1}{4},h_a=h_b=\frac{1}{4}$ &    7.219848e-09	&  ---&  ---\\ \vspace{0.1 cm}
  $$  & $\tau=\frac{\sqrt{2}}{16},h_a=h_b=\frac{1}{8}$&    8.823037e-10	&	2.0217	&	3.0326\\ \vspace{0.1 cm}
  $$& $\tau=\frac{1}{32},h_a=h_b=\frac{1}{16}$ &      1.171206e-10		&1.9422		&2.9133\\ \vspace{0.1 cm}
$$  & $\tau=\frac{\sqrt{2}}{128},h_a=h_b=\frac{1}{32}$&   1.445519e-11	&	2.0122	&	3.0183\\\hline\vspace{0.1 cm}
  $\alpha=1.7,\beta=1.4$& $\tau=\frac{1}{4},h_a=h_b=\frac{1}{4}$ &    7.181389e-09			&  ---&  ---\\ \vspace{0.1 cm}
  $$  & $\tau=\frac{\sqrt{2}}{16},h_a=h_b=\frac{1}{8}$&   8.771397e-10	  & 	2.0223	  & 	3.0334\\ \vspace{0.1 cm}
  $$& $\tau=\frac{1}{32},h_a=h_b=\frac{1}{16}$ &    1.165997e-10	  & 	1.9408	  & 	2.9112\\ \vspace{0.1 cm}
$$  & $\tau=\frac{\sqrt{2}}{128},h_a=h_b=\frac{1}{32}$&    1.439652e-11	  & 	2.0118	  & 	3.0178\\ \hline\vspace{0.1 cm}
 $\alpha=1.9,\beta=1.2$& $\tau=\frac{1}{4},h_a=h_b=\frac{1}{4}$ &   7.102704e-09		    &  ---&  ---\\ \vspace{0.1 cm}
  $$  & $\tau=\frac{\sqrt{2}}{16},h_a=h_b=\frac{1}{8}$&  8.628679e-10	& 	2.0274	& 	3.0412\\ \vspace{0.1 cm}
  $$& $\tau=\frac{1}{32},h_a=h_b=\frac{1}{16}$ &    1.150916e-10	& 	1.9376		& 2.9064\\ \vspace{0.1 cm}
$$  & $\tau=\frac{\sqrt{2}}{128},h_a=h_b=\frac{1}{32}$&   1.423779e-11	& 	2.0100	& 	3.0150\\\hline\vspace{0.1 cm}
\end{tabular}
 \end{footnotesize}
 \end{center}
 \end{table}

We solve this problem through method (\ref{eq36})--(\ref{eq38}) for different values of $\alpha,\beta$.
From Table \ref{Tab5}, one can see that the convergence orders of scheme (\ref{eq36})--(\ref{eq38}) are
$\mathcal{O}(\tau^2)$ in temporal direction and $\mathcal{O}(h_a^3+h_b^3)$ in spatial directions. It also coincides
with the theoretical analysis.

\section{Conclusion}
Based on the novel generating functions, we obtain several kinds of (generalized) high-order fractional-compact numerical approximation formulas
for Riemann-Liouville and/or Riesz derivatives with order lying in (1,2). For further checking the efficiency of these
high-order formulas, we apply the 3th-order formula to solve one- and two-dimensional
Riesz spatial fractional reaction dispersion equations. Both theoretical analysis and numerical tests show that the developed
numerical algorithms are efficient
and accurate.

%


  \end{document}